\documentclass[notitlepage, oneside]{article}
\usepackage{color}
\usepackage[pdftex]{hyperref}

\usepackage[initials]{amsrefs}
\usepackage[margin=1in]{geometry}

\usepackage{amsmath}
\usepackage{amsthm}	
\usepackage{graphicx}	
\usepackage{amsfonts}
\usepackage{amssymb}
\usepackage{bbm}
\usepackage{mathrsfs}

\theoremstyle{definition}
  \newtheorem{thm}{Theorem}
  \newtheorem{lemma}[thm]{Lemma}
  \newtheorem{prop}[thm]{Proposition}
  \newtheorem{cor}[thm]{Corollary} 
  
  \newtheorem{defn}[thm]{Definition}

  \newtheorem{rmk}[thm]{Remark}
  
  \newtheorem{thma}{Theorem}

\newcommand{\R}{\mathbb R}
\newcommand{\Z}{\mathbb Z}
\newcommand{\N}{\mathbb N}

\newcommand{\C}{\mathbb C}

\newcommand{\E}{\mathbb{E}}
\newcommand{\I}{\infty}

\newcommand{\PB}{\mathbb{P}}

\newcommand{\nk}{\binom{v}{k}}
\newcommand{\ud}{\, \mathrm{d}}

\DeclareMathOperator{\real}{Re}
\DeclareMathOperator{\imag}{Im}

\begin{document}

\title{An Asymptotic Formula for the Number of Balanced Incomplete Block Design Incidence Matrices}
\author{Aaron M. Montgomery \\ Department of Mathematics and Computer Science, Baldwin Wallace University}
\maketitle

\begin{abstract}
We identify a relationship between a random walk on a certain Euclidean lattice and incidence matrices of balanced incomplete block designs, which arise in combinatorial design theory.  We then compute the return probability of the random walk and use it to obtain the asymptotic number of BIBD incidence matrices (as the number of columns increases). Our strategy is similar in spirit to the one used by de Launey and Levin to count partial Hadamard matrices.
\end{abstract}


\section{Introduction}

In this paper, we will explore a relationship between a particular family of random walks and a certain collection of combinatorially-defined matrices. These matrices, known as balanced incomplete block design incidence matrices, are important in combinatorial design theory. Our goal will be to estimate the number of these matrices as the number of columns increases. Although this task is quite difficult to do directly, it is possible to estimate the return probability of the random walk and then to exploit the connection between the two problems to estimate the number of the matrices. This tactic, adopted from \cite{wdl_levin}, is noteworthy because it does not require explicit construction of the matrices being studied.
\begin{defn} \label{bibd_im_def}
  We say that a $v \times t$ matrix populated with 1's and 0's is an \emph{incidence matrix of a balanced incomplete block design} if there are positive integers $k$ and $\lambda$ such that:
  \begin{itemize}
    \item each column has exactly $k$ 1's, and
    \item each pair of distinct rows has inner product $\lambda$, which is independent of the choice of the pair.
  \end{itemize}
\end{defn}
We will use BIBD as a shorthand for balanced incomplete block design. The parameter $t$ is commonly written in literature as $b$ instead; we use $t$ here due to the connection to random walks in this paper. It is well-known that if $\lambda > 0$, the above conditions imply that the number of 1's in each row is a constant, which we will call $r$.  The following relations between $v, t, k, r, $ and $\lambda$ are also well-known:
  \begin{align}
	tk & = vr \label{BIBD_relation1}\\
	r(k-1) & = \lambda (v-1) \label{BIBD_relation2} \\
	tk (k-1) & = \lambda v (v-1) \label{BIBD_relations}
  \end{align}
A reference for \eqref{BIBD_relation1} and \eqref{BIBD_relation2} can be found at \cite{dinitz}*{p.\ 2}; from these, one can easily derive \eqref{BIBD_relations}.  

The task of analyzing and enumerating BIBDs has been well-studied; see, for instance, \cites{ostergard,dm,rowlinson,ok} for a few of many examples. A good summary of this work can be found in \cite{crc}*{ch. 1}. Often, the focus is on the underlying designs, rather than on the incidence matrices themselves. The primary difference between the two approaches is that when studying the matrices, we will regard matrices that differ only by a permutation of the columns as distinct, but if counting the underlying designs, such matrices would only count once since they correspond to the same design. Additionally, two designs are regarded to be isomorphic if their incidence matrices are the same up to a permutation of rows and columns (and are often similarly only counted once), but such matrices are regarded to be distinct for our purposes. 

The problems of counting the designs and of counting the incidence matrices (in the fashion described above) are certainly related, but translating between the two is a matter of some nontrivial combinatorics. Much of the aforementioned progress in enumerating BIBDs relies on computer algorithms, whereas this paper is purely analytic in nature. These problems are both difficult, and for a given set of parameters $v, t, k, r, \lambda$, even establishing the existence of such a design (or an incidence matrix) is often nontrivial. The conditions from \eqref{BIBD_relation1}, \eqref{BIBD_relation2}, and \eqref{BIBD_relations} are known to be necessary but not sufficient. For instance, the Bruck-Ryser-Chowla Theorem shows that no BIBD exists when $v = t = 22$, $k = r = 7$, and $\lambda = 2$. \cite{cr}*{Theorem 3} A computer search has shown the nonexistence of a BIBD with $v = 46,$ $t = 69,$ $k = 6,$ $r = 9,$ and $\lambda = 1$. \cite{htjl} On the other hand, it is known that BIBDs with $v = 12$, $t = 44$, $k = 3$, $r = 11$, and $\lambda = 2$ exist, and their enumeration has been carried out by a computer. \cite{ostergard} The complexity of this problem motivates our attempt to find an asymptotic formula for counting BIBD incidence matrices.

We note from relations \eqref{BIBD_relation1}, \eqref{BIBD_relation2}, and \eqref{BIBD_relations} that the parameters $v, k, t$ determine $r$ and $\lambda$, so we will use $v, k, t$ as our free parameters. The goal of this paper will be to prove the following theorem:

  \begin{thma} \label{number_BIBD_matrices}
  Let $v, k, t$ be such that $k \geq 2$, $v-k \geq 2$, $t \frac{k}{v} \in \Z$, and $t \frac{k(k-1)}{v(v-1)} \in \Z$.  Let $\Psi_{v,k,t}$ be the number of BIBD incidence matrices of dimensions $v \times t$ with $k$ 1's in each column, let $d = \binom{v}{2}$, and let
  \[ f(v,k) = 2 \left( \frac{(k-1)(v-k-1)}{v-3} \right)^{d-v} \left( \frac{k(v-k)}{v-2} \right)^{d-1} \left( \frac{1}{v(v-1)} \right)^{d} .\]
Then for fixed $v$ and $k$,
  \[\Psi_{v,k, t} = [1 + o(1)] \frac{\nk^t}{\sqrt{(2 \pi t)^{d-1} f(v,k)}} \textrm{ as } t \to \I . \]
  \end{thma}

To prove Theorem \ref{number_BIBD_matrices}, we will randomly generate matrices of suitable dimensions as follows: for fixed $v$ and $k$, we define $V_{v,k}$ to be the collection of all vectors in $\R^v$ with $k$ 1's and $v - k$ 0's.  We will construct a BIBD incidence matrix by concatenating randomly-drawn columns from the collection $V_{v,k}$ and considering whether the inner product condition is satisfied for the resulting matrix. This approach is convenient because it avoids the difficult task of actually constructing the matrices (see \cites{morales, abel, vantrung} for a few of the myriad examples of that type of work).

We now define our random walk and explain its correspondence with BIBD incidence matrices.  For an integer $v \geq 2$, we set $d = \binom{v}{2}$; the random walk will occur in $\R^d$, which will be regarded as a set of column vectors.  Instead of using the standard index system for coordinates of $\R^d$ (i.e. $1, \dots, d$), we will take our index set to be the set of all $S \subset \{1, \dots, v\}$ with $|S| = 2$.  We will write the components of $\vec x \in \R^d$ in lexicographic order; that is,
   \[ \vec x = (x_{\{1,2\}}, x_{\{1,3\}}, \dots, x_{\{v-2,v\}}, x_{\{v-1,v\}})^T. \]
We define a function $Z: V_{v,k} \to \R^d$ by
  \[ Z(\vec y) = (y_1 y_2, y_1 y_3, \dots, y_{v-2} y_v, y_{v-1} y_v)^T. \]
The purpose of this function is that if $Y$ is the $v \times t$ matrix with columns $\vec y^{(1)}, \dots, \vec y^{(t)}$, written as $Y = [\vec y^{(1)} \dots \vec y^{(t)}]$, and $\vec 1$ is the vector of all ones, then
  \[ Z(\vec y^{(1)}) + \dots + Z(\vec y^{(t)}) = \lambda \vec 1 \]
if and only if the inner product between any two rows of $Y$ is $\lambda$.  This allows us to reframe the constraint about the inner product or rows as one of a vector sum, which permits us to consider the problem in terms of a random walk.

\begin{defn}
  Let $\{X_t\}$ be the random walk on $\Z^d$ with increments drawn uniformly at random from $\{Z(\vec y): \vec y \in V_{v,k}\}$.
\end{defn}

From the previous discussion, the existence of a BIBD incidence matrix is then equivalent to the entry of the random walk $X_t$ into the diagonal set $\Delta = \{\lambda \vec 1: \lambda \in \Z\}$.  The random walk $X_t$ is not the ideal random walk to consider, for two reasons: first, the set $\Delta$ is infinite, which makes the probability that $X_t$ enters it a bit complicated.  Second, the increments of $X_t$ clearly do not have mean $\vec 0$, since vectors of the form $\{Z(\vec y): \vec y \in V_{v,k}\}$ also have entries that are only $0$ and $1$.

To fix these issues with $X_t$, we introduce a new random walk, $Y_t$, which will be $X_t$ with a correction for its drift.  If a vector is chosen uniformly from $\{Z(\vec y): \vec y \in V_{v,k} \}$, then the probability of coordinate $\{i,j\}$ being $1$ is equal to the probability that $y_i = 1$ and $y_j = 1$.  This probability is $\binom{v-2}{k-2}/\nk = \frac{k(k-1)}{v(v-1)}$, so to obtain a centered random walk, we subtract this term from each coordinate of the increments.  That is,
  \begin{equation} \label{defn_of_y}
    Y_t = X_t - \frac{k(k-1)}{v(v-1)} t \vec 1. 
  \end{equation}
Since we are interested in the probability that the random walk $X_t$ is equal to $\lambda \vec 1$ for some constant $\lambda$, we notice by \eqref{BIBD_relations} that $\lambda = \frac{k(k-1)}{v(v-1)} t$, which implies that $X_t = \lambda \vec 1$ if and only if $Y_t = 0$.  Hence, our tactic will be to estimate the probability that the random walk $Y_t$ returns to $\vec 0$ after $t$ steps.

Because $v \times t$ matrices populated with columns from $V_{v,k}$ lie in a 1-1 correspondence with paths of the random walk $X_t$ (hence, with paths of $Y_t$), it follows that
\[ \frac{\# \textrm{ BIBD incidence matrices} }{\# \textrm{ total matrices }} = \frac{\# \textrm{ return paths of $Y_t$ to $\vec 0$}}{\# \textrm{ all paths of $Y_t$}}. \]
The right-hand side of this equation is precisely the probability that the random walk $Y_t$ returns to $0$, which we will denote by $\PB_{v,k}^{(t)}(\vec 0, \vec 0)$.  (Our random walks are understood to always start at the origin, and the $v,k$ subscript serves only to indicate the preset parameters $v$ and $k$.)  The denominator of the left-hand side is $\nk^t$, since there are $\nk$ distinct choices for each of the $t$ columns.  Thus,
  \begin{equation} \label{prob_number_relation}
    \# \textrm{ BIBD incidence matrices} = \nk^t \PB_{v,k}^{(t)}(\vec 0, \vec 0)
  \end{equation}
so to prove Theorem \ref{number_BIBD_matrices}, we need only to find sufficiently accurate estimates on the return probability of the random walk $Y_t$.  We will accomplish this by proving a local central limit theorem for the quantity $\PB_{v,k}^{(t)}(\vec 0, \vec 0)$.

The basic strategy for estimating $\PB_{v,k}^{(t)}(\vec 0, \vec 0)$ will be the standard tactic of using the Fourier inversion formula (see, for instance, \cite{spitzer}*{P3, p.\ 57}).  Using the characteristic function $\Phi_Y: \R^d \to \C$, defined as
  \[\Phi_Y(\vec \theta) = \E[e^{i \vec \theta \cdot Y_1}] = \sum_{\vec y \in V_{v,k}} \nk^{-1} e^{i \vec \theta \cdot \left(Z(\vec y) - \frac{k(k-1)}{v(v-1)} \vec 1 \right)} \]
then for values $t$ such that $Y_t$ is supported on $\Z^d$, the return probability can be calculated as
  \begin{equation} \label{fourier_inversion}
    \PB_{v,k}^{(t)}(\vec 0, \vec 0) = \frac{1}{(2 \pi)^d} \int_{[-\pi, \pi]^d} \Phi_Y(\vec \theta)^t \ud \vec \theta \, .
  \end{equation}
We note that $Y_t$ is supported on $\Z^d$ if and only if $t \frac{k(k-1)}{v(v-1)} \in \Z$. However, it is not necessary to use the inversion formula in the case when $t\frac{k(k-1)}{v(v-1)} \not \in \Z$, since by \eqref{BIBD_relations} we trivially see that no such BIBD incidence matrix can exist.

To estimate the integral in \eqref{fourier_inversion}, we will divide $[-\pi, \pi]^d$ into regions where $|\Phi_Y(\vec \theta)|$ is close to $1$ and those where it is not, and we will provide estimates on $\Phi_Y(\vec \theta)$ accordingly.  As $t$ becomes large, the bulk of the integral will be determined by the regions in $\R^d$ where $|\Phi_Y(\vec \theta)|$ is close to 1, and the contributions from the other parts will become negligible. Since the random walk $Y_t$ is merely a spatially-shifted version of $X_t$, it will also be useful to consider the analogously-defined characteristic function $\Phi_X(\vec \theta) = \E[e^{i \vec \theta \cdot X_1}]$; we will explore the connections between the two and will switch our focus between $\Phi_X$ and $\Phi_Y$ depending on which is more convenient.

Before the proof of Theorem \ref{number_BIBD_matrices}, we note that the restrictions that $k \geq 2$ and $v-k \geq 2$ occur for technical reasons, although if $k = 1$, the BIBD incidence matrices are trivial in the sense that the inner product of any two distinct rows of any such matrix is automatically 0.  The case where $k = 2$ is nearly trivial as well, since a BIBD incidence matrix with $k = 2$ can only occur when every possible column from $V_{v,k}$ occurs the same number of times.  One can see without any advanced tactics that the number of such matrices must then be
\[ \Psi_{v, 2, t} = \frac{t!}{ [(t/d)!]^d}\]
which is asymptotically equivalent to the formula in Theorem \ref{number_BIBD_matrices} as shown by Stirling's formula.

We also remark that while in principle the calculation of the return probability of $Y_t$ is just a  matter of computing asymptotic values in a local central limit theorem, the walk has a special structure that complicates matters.  In particular, the increment set of the walk is not symmetric, and the walk takes place on a sublattice of $\R^d$ which is difficult to specify as a purely combinatorial entity.  For these reasons, the common approach of explicitly transforming the walk $Y_t$ to a strongly aperiodic random walk on an integer lattice is challenging here, and we will instead opt for the Fourier-analytic approach as previously outlined.



The outline of the sections is as follows: in Section \ref{S:lambda_max}, we give an explicit description of the so-called `maximal set' of the characteristic function; that is, the set where $|\Phi_Y(\vec \theta)| = 1$.  In Section \ref{S:multiset}, we discuss how to decompose the integral in \eqref{fourier_inversion} in terms of this maximal set.  In Section \ref{S:upper_bounds}, we provide estimates on the integral contributions far from the maximal set.  In Section \ref{S:near_lambda}, we introduce an important combinatorially-defined matrix $N$ and use it to obtain bounds on the integral contribution near the maximal set.  In Section \ref{S:matrix_determinant}, we compute the expression $f(v,k)$ found in the statement of Theorem \ref{number_BIBD_matrices}.  This expression will arise as the determinant of a principal submatrix of the aforementioned matrix $N$.  Finally, in Section \ref{S:proof} we put all the parts together to prove Theorem \ref{number_BIBD_matrices}.

\section{Extreme Values of the Characteristic Function} \label{S:lambda_max}

In this section, we seek to understand the set where the characteristic functions $\Phi_X$ and $\Phi_Y$ have maximum absolute value.  We begin with the operative definitions:
  \begin{align*}
	\Lambda_X & =  \{\vec \theta \in [-\pi, \pi]^d : |\Phi_X(\vec \theta)| = 1 \} \\
	\Lambda_Y & =  \{\vec \theta \in [-\pi, \pi]^d : |\Phi_Y(\vec \theta)| = 1 \} 
  \end{align*}

\begin{prop} \label{sets_are_equal}
  The sets $\Lambda_X$ and $\Lambda_Y$ are equal.
\end{prop}

\begin{proof}
Note that $Y_1 = X_1 - \vec u$, where $\vec u$ is deterministic.  Then for any $\vec \theta$,
  \begin{align*}
  	| \Phi_Y(\vec \theta)| & =  |\E[e^{i \vec \theta \cdot (X_1 - \vec u)}]| \\
	& =  |\E[e^{i \vec \theta \cdot X_1}]| \cdot |e^{-i \vec \theta \cdot \vec u}| \\
	& =  |\Phi_X(\vec \theta)|
  \end{align*}
which gives the desired result.
\end{proof}

Although $\Lambda_Y$ corresponds to the random walk actually used in the calculation and in the Fourier Inversion formula in \eqref{fourier_inversion}, $\Lambda_X$ corresponds to the walk without the drift correction and is at times more computationally convenient.  We note that
\[\vec \ell \in \Lambda_X \iff e^{i \vec \ell \cdot Z(\vec x)} = e^{i \vec \ell \cdot Z(\vec y) } \textrm{ for all } \vec x, \vec y \in V_{v,k}\]
which implies that
\begin{equation} \label{modulo_condition}
	\vec \ell \in \Lambda_X \iff \textrm{ for all }\vec x, \vec y \in V_{v,k}, \ \vec \ell \cdot Z(\vec x) \equiv \vec \ell \cdot Z(\vec y) \pmod {2 \pi} .
\end{equation}

\begin{prop} \label{char_function_properties}
  If $\vec \ell \in \Lambda_X$ and $\vec \gamma \in [-\pi, \pi]^d$, then $\Phi_X(\vec \ell + \vec \gamma) = \Phi_X(\vec \ell) \Phi_X(\vec \gamma)$ and $\Phi_Y(\vec \ell + \vec \gamma) = \Phi_Y(\vec \ell)\Phi_Y(\vec \gamma)$.
\end{prop}

\begin{proof}
  Let $\vec \ell \in \Lambda_X$.  By \eqref{modulo_condition}, we see that $\vec \ell \cdot X_1$ does not depend on the random vector $X_1$, so $e^{i \vec \ell \cdot X_1}$ is a deterministic quantity.  Hence, 
  \begin{align*}
	  \Phi_X(\vec \ell + \vec \gamma) & =  \E[e^{i(\vec \ell + \vec \gamma) \cdot X_1}]\\
	  & = e^{i \vec \ell \cdot X_1} \E[e^{\vec \gamma \cdot X_1}]
  \end{align*}
and since $e^{i \vec \ell \cdot X_1} = \E[e^{i \vec \ell \cdot X_1}]$, the first claim is shown.  The proof of the same statement for $\Phi_Y$ is identical.
\end{proof}

\begin{rmk} \label{lambda_is_closed}
  In particular, we see that $\Lambda_X$ is closed under addition modulo $2 \pi$.  Moreover, \eqref{modulo_condition} shows that $\Lambda_X$ is closed under negation, so it is closed under subtraction (modulo $2 \pi$) as well.
\end{rmk}

In all the following, we will assume that $k \geq 2$ and $v-k \geq 2$. We will frequently refer to vectors in $\R^d$ being equivalent modulo $2 \pi$; by this, we mean that all their corresponding coordinates should be congruent to one another modulo $2 \pi$.

\begin{lemma} \label{near_lambda_1}
  Let $\vec \mu \in [-\pi, \pi]^d$.  Suppose that there exists $\epsilon_0 > 0$ such that for all $\vec x, \vec y \in V_{v,k}$, there exist $z \in \Z$ and $\epsilon$ with $|\epsilon| < \epsilon_0$ such that $[Z(\vec x) \cdot \vec \mu - Z(\vec y) \cdot \vec \mu] = 2 \pi z + \epsilon$.   Then for any distinct integers $a, b, c, d \in \{1, \dots, v\}$ there exist $z \in \Z$ and $\epsilon$ with $|\epsilon| < 2 \epsilon_0$ such that $[\mu_{\{a,c\}} - \mu_{\{b,c\}}] = [\mu_{\{a,d\}} - \mu_{\{b,d\}}] + 2 \pi z + \epsilon$.
\end{lemma}

The interpretation of this lemma is that if $[Z(\vec x) \cdot \vec \mu - Z(\vec y) \cdot \vec \mu] \bmod 2\pi$ is nearly 0 for all $\vec x, \vec y \in V_{v,k}$, then expressions of the form $[\mu_{\{a,j\}} - \mu_{\{b,j\}}] \bmod 2\pi$ are (nearly) independent of $j$. 

After establishing Lemma \ref{near_lambda_1}, we obtain a useful corollary by letting $\epsilon_0 \to 0$ and using \eqref{modulo_condition}:

\begin{cor} \label{lambda_obs_1}
  If $\vec \ell \in \Lambda_X$, then for any fixed $a,b$ the expression $\ell_{\{a,j\}} - \ell_{\{b,j\}}$ is independent of $j$ $\pmod {2 \pi}$.
\end{cor}

We remark that the original idea for Corollary \ref{lambda_obs_1} was suggested by Warwick de Launey in a personal communication via David Levin.

\begin{proof}[Proof of Lemma \ref{near_lambda_1}]

Without loss of generality, we assume that $a = 1$, $b = 2$, $c = 3$, and $d = 4$. We first define the following vectors in $V_{v,k}$:
  \begin{align*}
  	\vec x_1 & = (1, 0, 1, 0, \overbrace{1, \dots, 1}^{k-2}, \overbrace{0, \dots, 0}^{v-k-2})^T \\
	\vec x_2 & = (1, 0, 0, 1, 1, \dots, 1, 0, \dots, 0)^T \\
	\vec x_3 & = (0, 1, 1, 0, 1, \dots, 1, 0, \dots, 0)^T \\
	\vec x_4 & = (0, 1, 0, 1, 1, \dots, 1, 0, \dots, 0)^T
  \end{align*}
These vectors are identical except in the first four coordinates.  For any $\vec \mu \in [-\pi, \pi]^d$, we have
  \begin{align*}
  	\vec \mu \cdot Z(\vec x_1) & = \mu_{\{1,3\}} + \sum_{j=5}^{k+2} \mu_{\{1,j\}} + \sum_{j=5}^{k+2} \mu_{\{3,j\}} + \sum_{5 \leq i < j \leq k+2} \mu_{\{i,j\}} \\
  	\vec \mu \cdot Z(\vec x_2) & = \mu_{\{1,4\}} + \sum_{j=5}^{k+2} \mu_{\{1,j\}} + \sum_{j=5}^{k+2} \mu_{\{4,j\}} + \sum_{5 \leq i < j \leq k+2} \mu_{\{i,j\}} \\
  	\vec \mu \cdot Z(\vec x_3) & = \mu_{\{2,3\}} + \sum_{j=5}^{k+2} \mu_{\{2,j\}} + \sum_{j=5}^{k+2} \mu_{\{3,j\}} + \sum_{5 \leq i < j \leq k+2} \mu_{\{i,j\}} \\
  	\vec \mu \cdot Z(\vec x_4) & = \mu_{\{2,4\}} + \sum_{j=5}^{k+2} \mu_{\{2,j\}} + \sum_{j=5}^{k+2} \mu_{\{4,j\}} + \sum_{5 \leq i < j \leq k+2} \mu_{\{i,j\}}	
  \end{align*}
and hence,
\begin{align*}
 & \vec \mu \cdot [Z(\vec x_1) - Z(\vec x_2)] + \vec \mu \cdot [Z(\vec x_4) - Z(\vec x_3)] \\
 & \qquad \quad = \vec \mu \cdot [Z(\vec x_1) - Z(\vec x_2) - Z(\vec x_3) + Z(\vec x_4)] \\
 & \qquad \quad = \mu_{\{1,3\}} + \mu_{\{2,4\}} - \mu_{\{1,4\}} - \mu_{\{2,3\}}. 
\end{align*}

Our assumption implies that there exist $z \in \Z$ and $\epsilon_1 \in (-2 \epsilon_0, 2 \epsilon_0)$ such that
  \begin{equation*}
    [\mu_{\{1,3\}} - \mu_{\{2,3\}}] = [\mu_{\{1,4\}} - \mu_{\{2,4\}}] + \epsilon_1 + 2 \pi z
  \end{equation*}
by the triangle inequality.  

To verify that generality was not lost in our above argument, we note that if $a, b, c, d$ were arbitrary and distinct, we could permute the coordinates of the $\vec x_1, \vec x_2, \vec x_3, \vec x_4$ appropriately and repeat the same argument. The essential point is that these four vectors are identical in all but the coordinates $a$, $b$, $c$, and $d$, and that their differences in those coordinates parallel the ones above. After this adjustment, the proof proceeds as above. \end{proof}

\begin{lemma} \label{near_lambda_2}
  Let $\vec \mu \in [-\pi, \pi]^d$.  Suppose that there exists $\epsilon_0 > 0$ such that for all $\vec x, \vec y \in V_{v,k}$, there exist $z \in \Z$ and $\epsilon$ with $|\epsilon| < \epsilon_0$ such that $[Z(\vec x) \cdot \vec \mu - Z(\vec y) \cdot \vec \mu] = 2 \pi z + \epsilon$. Then for all $a, b, c, d$, there exists $z \in \Z$ and $\epsilon$ with $|\epsilon | < 4 \epsilon_0$ such that $[\mu_{\{a,b\}} - \mu_{\{c,d\}} ]= \frac{2 \pi}{k-1} z + \epsilon$.
\end{lemma}

The interpretation of this lemma is that if $[Z(\vec x) \cdot \vec \mu - Z(\vec y) \cdot \vec \mu] \bmod 2\pi$ is nearly 0 for all $\vec x, \vec y \in V_{v,k}$, then all vector components of $\vec \mu$ are nearly constant modulo $\frac{2 \pi}{k-1}$. We note that this lemma implicitly requires that $a \neq b$ and $c \neq d$.

  As before, Lemma \ref{near_lambda_2} yields a useful corollary obtained by letting $\epsilon_0 \to 0$ and using \eqref{modulo_condition}:

\begin{cor} \label{lambda_obs_2}
  If $\vec \ell \in \Lambda_X$, then all the components of $\vec \ell$ are congruent to one another $\pmod {\frac{2 \pi}{k-1}}$.
\end{cor}

\begin{proof}[Proof of Lemma \ref{near_lambda_2}]
We define some vectors from $V_{v,k}$:
  \begin{align*}
  	\vec y_1 & = (1, 0, \overbrace{1, \dots, 1}^{k-1}, \overbrace{0, \dots, 0}^{v-k-1})^T \\
	\vec y_2 & = (0, 1, 1, \dots, 1, 0, \dots, 0)^T
  \end{align*}
  These vectors are identical except in the first two coordinates.  For any $\vec \mu$, we have
  \begin{align*}
  	\vec \mu \cdot Z(\vec y_1) & = \sum_{j=3}^{k+1} \mu_{\{1,j\}} \\
	\vec \mu \cdot Z(\vec y_2) & = \sum_{j=3}^{k+1} \mu_{\{2,j\}}
  \end{align*}
so by assumption, we then have $z \in \Z$ and $\epsilon_1 \in (-\epsilon_0, \epsilon_0)$ such that 
  \begin{align*}
    \vec \mu \cdot [Z(\vec y_1) - Z(\vec y_2)] & = \sum_{j=3}^{k+1} [\mu_{\{1,j\}} - \mu_{\{2,j\}}] \\
    & = 2 \pi z + \epsilon_1 .
  \end{align*}
Next, we fix some integer $n$ with $3 \leq n \leq v$.  For each term in the sum where $j \neq n$, we use Lemma \ref{near_lambda_1} to replace $[\mu_{\{1,j\}} - \mu_{\{2,j\}}]$ with $[\mu_{\{1,n\}} - \mu_{\{2,n\}}]$ plus an error term.  Executing this replacement for all $j$ shows that there exist $z \in \Z$ and $\epsilon_2$ with $|\epsilon_2| < 2(k-1) \epsilon_0$ such that
  \begin{equation*}
    (k-1)[\mu_{\{1,n\}} - \mu_{\{2,n\}}] = 2 \pi z + \epsilon_2 .
  \end{equation*}
Dividing by $k-1$ then shows that there exists $\epsilon_3$ with $|\epsilon_3| < 2 \epsilon_0$ such that
  \begin{equation*} 
  \mu_{\{1,n\}} - \mu_{\{2,n\}} = \frac{2\pi}{k-1} z + \epsilon_3.
  \end{equation*}

We note here that the choices of $1$ and $2$ in the coordinates of $\mu$ were merely consequences of the construction of the vectors $\vec y_1$ and $\vec y_2$.  For any distinct $f, g, h$, permuting the coordinates of those vectors appropriately (and adjusting the subsequent arguments) shows that there exist $z \in Z$ and $\epsilon_3$ with $|\epsilon_3| < 2 \epsilon_0$ such that
  \begin{equation} \label{lemmette_fuzz_3}
    \mu_{\{f,h\}} - \mu_{\{g,h\}} = \frac{2\pi}{k-1} z + \epsilon_3.
  \end{equation}

Finally, we let $a, b, c, d$ be distinct.  By applying \eqref{lemmette_fuzz_3} twice and using the triangle inequality, we see that there exist $z \in \Z$ and $\epsilon_4$ with $|\epsilon_4| < 4 \epsilon_0$ such that
  \begin{align*}
    \mu_{\{a,b\}} - \mu_{\{c,d\}} & = [\mu_{\{a,b\}} - \mu_{\{a, d\}}] + [\mu_{\{a,d\}} - \mu_{\{c,d\}}] \\
    & = \frac{2 \pi z}{k-1} + \epsilon_4
  \end{align*}
as desired.
\end{proof}

Next, we examine some ``building block'' vectors that will help to characterize the set $\Lambda_X$.  

\begin{defn} \label{defn_alpha_beta}
For a fixed $v$ and $k$ and $1 \leq a \leq v$, we define the vector $\vec \beta^a$ to be the vector with $\beta^a_{\{i,j\}} = 1$ if $i = a$ or $j = a$, and $\beta^a_{\{i,j\}}=0$ otherwise.  We also define $\vec \alpha^a = \vec 1 - \vec \beta^a$; that is, $ \alpha^a_{\{i,j\}} = 0$ if $i = a$ or $j = a$ and $ \alpha^a_{\{i,j\}} = 1$ otherwise.
\end{defn}

\begin{prop} \label{in_lambda}
  The vectors $\frac{2 \pi}{k-1} \vec \beta^a$ and $\frac{2 \pi}{k-1} \vec \alpha^a$ are in $\Lambda_X$.  Moreover, so also is $\gamma \vec 1$ for any real $\gamma$.
\end{prop}

\begin{proof}
  In light of \eqref{modulo_condition}, we wish to show that modulo $2 \pi$, the expressions $\frac{2 \pi}{k-1} \vec \beta^a \cdot Z(\vec x)$ and $\frac{2 \pi}{k-1} \vec \alpha^a \cdot Z(\vec x)$ do not depend on the choice of $\vec x \in V_{v,k}$.
  
Fix $a$, and let $\vec x \in V_{v,k}$ be arbitrary. A straightforward calculation shows that if $x_a = 1$, then $Z(\vec x) \cdot \frac{2 \pi}{k-1} \vec \beta^a = (k-1) \cdot \frac{2 \pi}{k-1}$, since $Z(\vec x)$ will have exactly $k-1$ coordinates of the form $\{a, \cdot\}$ whose entries are 1. On the other hand, if $x_a = 0$, then $Z(\vec x) \cdot \frac{2\pi}{k-1} \vec \beta^a = 0$. This establishes that $\frac{2 \pi}{k-1} \vec \beta^a \in \Lambda_X$. 

A similarly straightforward calculation shows that if $x_a = 1$, then $Z(\vec x) \cdot \frac{2 \pi}{k-1} \vec \alpha^a = [\binom{k}{2} - (k-1)] \frac{2\pi}{k-1}$, and if $x_a = 0$, then $Z(\vec x) \cdot \frac{2 \pi}{k-1} \vec \alpha^a = \binom{k}{2} \frac{2 \pi}{k-1}$. This shows that $\frac{2 \pi}{k-1} \vec \alpha^a \in \Lambda_X$.

Finally, for any $\vec x \in V_{v,k}$ and any $\gamma \in \R$, we have $Z(\vec x) \cdot \gamma \vec 1 = \binom{k}{2} \gamma$.
\end{proof}

Using these vectors, we arrive at the desired full characterization of $\Lambda_X$.

\begin{lemma} \label{lambda_structure}
  Let $\vec \beta^a$ and $\vec \alpha^a$ be as defined in Definition \ref{defn_alpha_beta}. Suppose that $\vec \ell \in [-\pi, \pi)^d$ and $\vec \ell \in \Lambda_X$.  Then there exist $\gamma \in [0, 2 \pi)$ and integers $m_i \in [0, k-1)$ such that
  \[ \vec \ell \equiv \gamma \vec 1 + m_1\frac{2 \pi}{k-1} \vec \alpha^1+ \sum_{j=3}^{v} m_j \frac{2 \pi}{k-1} \vec \beta^j \pmod{2 \pi}. \]
Moreover, this representation of $\vec \ell$ is unique.
\end{lemma}

\begin{rmk} \label{lambda_is_lines}
  This decomposition of $\Lambda_X (= \Lambda_Y)$ shows that the set is made up of a number of distinct $1$-dimensional sets, all of which are lines parallel to the vector $\vec 1$.  
\end{rmk}

\begin{proof}[Proof of Lemma \ref{lambda_structure}]

Let $\ell \in \Lambda_X$. Set $\gamma = \ell_{\{1,2\}}$, and set $\vec \theta = \vec \ell - \gamma \vec 1$. We note that by Remark \ref{lambda_is_closed} and Proposition \ref{in_lambda} that $\vec \theta \in \Lambda_X$. Moreover, since $\theta_{\{1,2\}} = 0$, by Corollary \ref{lambda_obs_2} we see that $\theta_{\{a,b\}} \equiv 0 \pmod {\frac{2 \pi}{k-1}}$ for all $\{a,b\}$. That is, for each $\{a,b\}$, there are unique integers $z$ and $m$ (both of which depend on $a$ and $b$) such that $m \in [0, k-1)$ and
  \begin{equation} \label{lambda_structure_accounting}
    \theta_{\{a,b\}} = 2 \pi z + \frac{2 \pi}{k-1} m.
  \end{equation}
 For $j \geq 3$, we set $m_j$ to be the integer $m$ which satisfies \eqref{lambda_structure_accounting} when $\{a,b\} = \{1,j\}$. If we set
   \[ \vec \zeta = \vec \theta - \sum_{j=3}^{v} m_j \frac{2 \pi}{k-1} \vec \beta^j,\]
then we again note by Remark \ref{lambda_is_closed} and Proposition \ref{in_lambda} that $\vec \zeta \in \Lambda_X$. We still have $\zeta_{\{a,b\}} \equiv 0 \pmod{\frac{2\pi}{k-1}}$ for all $\{a,b\}$. Hence, to complete the existence portion of the proof, it remains only to show that $\vec \zeta$ is a multiple of $\vec \alpha^1$.

For a fixed $j \geq 2$, the only vector of the set $\{\vec \beta^i: i \geq 2\}$ with a nonzero $\{1,j\}$ component is $\vec \beta^j$. Thus, $\zeta_{\{1,j\}} \equiv 0 \pmod{2 \pi}$ for all $j \geq 3$; further, since $\theta_{\{1,2\}} \equiv 0 \pmod{2 \pi}$, we have $\zeta_{\{1,j\}} \equiv 0 \pmod{2 \pi}$ as well. For $3 \leq i < j \leq v$, by Corollary \ref{lambda_obs_1} we have 
  \begin{align*}
    \zeta_{\{2,j\}} - \zeta_{\{2,3\}} &\equiv \zeta_{\{1,j\}} - \zeta_{\{1,3\}} \pmod{2 \pi}, \\
    \zeta_{\{i,j\}} - \zeta_{\{1,i\}} &\equiv \zeta_{\{2,j\}} - \zeta_{\{1,2\}} \pmod{2 \pi}.
  \end{align*}
Since $\zeta_{\{1,2\}} \equiv \zeta_{\{1,3\}} \equiv \zeta_{\{1,i\}} \equiv \zeta_{\{1,j\}} \equiv 0 \pmod{2 \pi}$, these equations imply that $\zeta_{\{i,j\}} \equiv \zeta_{\{2,j\}} \equiv \zeta_{\{2,3\}} \pmod{2 \pi}$, which shows that $\vec \zeta$ is a multiple of $\vec \alpha^1$.

Finally, we argue the uniqueness of these expressions of vectors in $\Lambda_X$. Of the collection of vectors consisting of $\vec 1, \vec \alpha^1$, and $\vec \beta^j$ with $j \geq 3$, only $\vec 1$ has a nonzero $\{1,2\}$ component; this implies the uniqueness of $\gamma$. Further, for $j \geq 3$, only $\vec 1$ and $\vec \beta^j$ have a nonzero $\{1,j\}$ component; this implies the uniqueness of $m_j$ for $j \geq 3$. The uniqueness of the final coefficient, $m_1$, follows.
\end{proof}
\section{Anatomy of the Integral} \label{S:multiset}

Having worked in the previous section to obtain a full characterization of the set $\Lambda_Y$, our next goal is to explain how we will decompose the integral in \eqref{fourier_inversion}.  The ultimate goal of this section will be to work toward the decompositions found in \eqref{integral_breakup_4} and \eqref{integral_breakup_5}.  These expression will require a good deal of technical setup.  The outline of this section is as follows: first, Lemma \ref{multiset_lemma} and Proposition \ref{bibd_conditions_walk} will explore the nature of the multi-set $\{\Phi_Y(\vec \ell)^t: \vec \ell \in \Lambda_Y\}$.  Next, we will discuss how we separate the region $[-\pi, \pi]^d$ into smaller pieces, culminating with \eqref{integral_breakup_3}.  Finally, we will combine these two ideas to obtain \eqref{integral_breakup_4} and \eqref{integral_breakup_5}.

We begin with the multi-set $\{\Phi_Y(\vec \ell)^t: \vec \ell \in \Lambda_Y\}$ and will first consider the case where $t = 1$.

\begin{lemma} \label{multiset_lemma}
Let $\vec \ell = \gamma \vec 1 + m_1 \frac{2 \pi}{k-1} \vec \alpha^1 + \sum_{j=3}^v m_j \frac{2 \pi}{k-1} \vec \beta^j$ be in $\Lambda_Y$, and define $S(\vec \ell) = m_1 - \sum_{j=3}^v m_j$.  Then $\Phi_Y(\vec \ell) = e^{ i \frac{2 \pi k}{v} S(\vec \ell)}$.
\end{lemma}

\begin{proof}
  We recall three computations from the proof of Proposition \ref{in_lambda}:
    \begin{align*}
      \vec 1 \cdot X_1 &= \binom k 2 \\
      \frac{2 \pi}{k-1} \vec \alpha^1 \cdot X_1 &\equiv \binom k 2 \frac{2 \pi}{k-1} \pmod{2 \pi} \\
      \frac{2 \pi}{k-1} \vec \beta^j \cdot X_1 &\equiv 0 \pmod{2 \pi}
    \end{align*}
  We also note from Definition \ref{defn_alpha_beta} that $\vec 1 \cdot \vec \alpha^a = \binom {v-1} 2$ and $\vec 1 \cdot \vec \beta^a = v-1$.
  From these, \eqref{defn_of_y}, and Definition \ref{defn_alpha_beta}, the following are easily verified and complete the proof:
    \begin{align*}
      \vec 1 \cdot Y_1 &= 0 \\
      \frac{2 \pi}{k-1} \vec \alpha^1 \cdot Y_1 & \equiv \frac{2 \pi k}{v} \pmod{2 \pi} \\
      \frac{2 \pi}{k-1} \vec \beta^j \cdot Y_1 &\equiv - \frac{2 \pi k}{v} \pmod{2 \pi}. \qedhere
    \end{align*}
\end{proof}

Our next goal is to investigate the nature of the multi-set 
  \[ \{ \Phi_Y(\vec \ell): \vec \ell \in \Lambda_Y \}. \] 
Because $\gamma$ in Lemma \ref{lambda_structure} can be anything in the interval $[0, 2\pi)$, the set $\Lambda_Y$ is infinite. We define the set
  \begin{equation} \label{lambda_diagonal_cosets}
    \Lambda_Y^{\square} = \left \{ m_1 \frac{2 \pi}{k-1} \vec \alpha^1 + \sum_{j=3}^v m_j \frac{2 \pi}{k-1} \vec \beta^j : m_i \in \Z \cap [0, k-1)\right \}
  \end{equation}
by eliminating the $\gamma \vec 1$ component of $\Lambda_Y$.  We also define the set 
  \begin{equation} \label{def_lambda_star}
    \Lambda_Y^{\star} = \left\{ \vec \psi \in [-\pi, \pi)^d : \vec \psi \equiv \vec \psi^{\square} \pmod {2 \pi} \textrm{ for some } \vec \psi^{\square} \in \Lambda_Y^{\square} \right\}.
  \end{equation}
We note that each vector in $\Lambda_Y^{\square}$ has a unique representative in $[-\pi, \pi)^d$.  Lemma \ref{multiset_lemma} shows that for any $\vec \psi \in \Lambda_Y$ and any $\gamma$, we have $\Phi_Y(\vec \psi + \gamma \vec 1) = \Phi_Y(\vec \psi)$.  Therefore, in order to understand the nature of the multi-set  $\{\Phi_Y(\vec \psi): \vec \psi \in \Lambda_Y\}$, it suffices to consider the multi-set $\{\Phi_Y(\vec \psi): \vec \psi \in \Lambda_Y^{\star}$\}.  This is particularly useful since $\Lambda_Y$ consists of several subsets parallel to $\vec 1$, whence the set $\Lambda_Y^{\star}$ consists of one representative vector for each distinct diagonal component.  It is easy to see that
  \begin{equation} \label{counting_lambda}
    | \Lambda_Y^{\star}| = (k-1)^{v-1}.
  \end{equation}

We remark here that since 
  \[Y_t = X_t - \frac{k(k-1)}{v(v-1)} t \vec 1 \]
and $X_t \in \Z^d$, the random walk $Y_t$ is supported on the lattice $\Z^d$ if and only if $t \frac{k(k-1)}{v(v-1)} \in \Z$.  Hence, the Fourier Inversion Formula in \eqref{fourier_inversion} only applies when $t \frac{k(k-1)}{v(v-1)} \in \Z$. For values of $t$ such that this is not the case, the probability that $Y_t = \vec 0$ is trivially $0$.  This constraint that $t \frac{k(k-1)}{v(v-1)} \in \Z$ corresponds to the BIBD constraint in \eqref{BIBD_relations}.  We also note by the BIBD constraint in \eqref{BIBD_relation1} that we must have $t \frac{k}{v} \in \Z$ as well, though this requirement manifests in a more subtle way than the necessity that $t \frac{k(k-1)}{v(v-1)} \in \Z$.  For certain choices of $k$ and $v$, such as $k = 3$ and $v = 5$, it holds that $t \frac{k(k-1)}{v(v-1)} \in \Z$ implies that $t \frac{k}{v} \in \Z$.  For other choices, such as $k = 3$ and $v = 6$, this is not the case.  Our next lemma will eventually be used to show how a positive return probability of the walk $Y_t$ intrinsically requires that $t \frac{k}{v} \in \Z$.

\begin{prop} \label{bibd_conditions_walk}
Suppose $t \frac{k(k-1)}{v(v-1)} \in \Z$.  
  \begin{itemize}
    \item If $t \frac{k}{v} \in \Z$, then the multi-set $\{ \Phi_Y(\vec \psi)^t: \vec \psi \in \Lambda_Y^{\star} \}$ consists only of the number $1$, repeated $(k-1)^{v-1}$ times.
    \item If $t \frac{k}{v} \not \in \Z$, then the multi-set $\{ \Phi_Y(\vec \psi)^t: \vec \psi \in \Lambda_Y^{\star} \}$ consists of all the powers of a certain root of unity, each appearing the same number of times; consequently, the sum of these roots is zero.
   \end{itemize}
\end{prop}

\begin{proof}
  
Suppose that $t \frac{k}{v} \in \Z$.  By Lemma \ref{multiset_lemma}, we have
  \[ \Phi_Y(\vec \psi)^t = e^{i 2 \pi t \frac{k}{v} S(\vec \psi)}\]
and since $t \frac{k}{v} S(\vec \psi) \in \Z$, it follows that $\Phi_Y(\vec \psi)^t = 1$ for all $\vec \psi \in \Lambda_Y$.

Next, suppose that $t \frac{k}{v} \not \in \Z$, but that $t \frac{k(k-1)}{v(v-1)} = j$ with $j \in \Z$.  In this case, we have $t \frac{k}{v} = \frac{j(v-1)}{k-1}$.  We can express this in a reduced form; i.e. $t \frac{k}{v} = \frac{a}{b}$ with $b | (k-1)$, $b \neq 1$, and $a$ relatively prime to $b$.  By examining \eqref{lambda_diagonal_cosets}, we see that the multi-set $\{S(\vec \psi) \bmod (k-1): \psi \in \Lambda_Y^{\star} \}$ consists of the numbers in $\{0, \dots, k-2\}$, counted $(k-1)^{v-2}$ times each.  Since 
  \[\Phi_Y(\vec \psi) = \exp \left( 2 \pi i \frac{a}{b} S(\vec \psi) \right)  \]
and $b | (k-1)$, it follows that the multi-set $ \{\Phi_Y(\vec \psi): \vec \psi \in \Lambda_Y^{\star} \}$ consists of all the $b^{th}$ roots of unity, each having the same number of appearances.
\end{proof}

We now seek to break up the integral $(2 \pi)^{-d} \int_{[-\pi, \pi]^d} \Phi_Y(\vec \theta)^t \ud \vec \theta$ into manageable pieces.  We define the set
  \begin{equation} \label{defn_lambda_not}
    \Lambda_0 = \left\{\vec \ell \in \R^d : \ell_{\{a,b\}} \equiv \ell_{\{c,d\}} \pmod {2\pi/(k-1)} \textrm{ for all } a, b, c, d  \right\}
  \end{equation}
and note by Corollary \ref{lambda_obs_2} that $\Lambda_X \subset \Lambda_0$. 

\begin{defn} \label{R_sets}
For $\delta > 0$, we partition $\R^d$ into three regions:
  \begin{align*}
    R_A^{\delta} & = \{ \vec \ell + \vec \zeta: \vec \ell \in \Lambda_X \textrm{ and } |\zeta_{\{i,j\}}| < \delta \textrm{ for all } i,j \} \\
    R_B^{\delta} & = \{ \vec \ell + \vec \zeta : \vec \ell \in \Lambda_0 \setminus \Lambda_X \textrm{ and } |\zeta_{\{i,j\}}| < \delta \textrm{ for all } i,j\} \\
    R_C^{\delta} & = \R^d \setminus (R_A^{\delta} \cup R_B^{\delta}).
  \end{align*}
\end{defn}

The idea is that $R_A^{\delta}$ is the region close to $\Lambda_X$. The set $R_B^{\delta}$ is the region which is close to satisfying the modular condition in \eqref{defn_lambda_not} but is not close to $\Lambda_X$. Finally, $R_C^{\delta}$ is the region which is far from satisfying the modular condition. Since $\Lambda_X$ is where the characteristic function has $|\Phi_Y(\vec \ell)| = 1$, only $R_A^{\delta}$ should significantly contribute to the integral in \eqref{fourier_inversion}, while the other terms should become negligible for sufficiently large $t$.

What follows are some technical observations about these newly-defined sets.

\begin{lemma} \label{differences_of_lambda_not}
  Suppose $\delta < \frac{\pi}{2(k-1)}$ and that $\vec \mu^1, \vec \mu^2 \in R_A^{\delta} \cup R_B^{\delta}$ have $\vec \mu^1 \equiv \vec \mu^2 \pmod{2 \pi}$. Let $\vec \mu^1 = \vec \ell^1 + \vec \zeta^1$ and $\vec \mu^2 = \vec \ell^2 + \vec \zeta^2$ as in Definition \ref{R_sets}. Then $\vec \ell^1 - \vec \ell^2$ is equivalent to a scalar multiple of $\vec 1 \pmod{2\pi}$.
\end{lemma}

\begin{proof}
We first note that $\vec \ell^1 - \vec \ell^2 \equiv \vec \zeta^2 - \vec \zeta^1 \pmod{2 \pi}$. If we set $\vec \theta = \vec \ell^1 - \vec \ell^2$, then since both $\vec \ell^1$ and $\vec \ell^2$ are in $\Lambda_0$, so also is $\vec \theta$. Let $a, b, c, d$ be arbitrary. Since $\vec \theta \in \Lambda_0$, it follows that $\theta_{\{a,b\}} - \theta_{\{c,d\}}$ is a multiple of $\frac{2 \pi}{k-1}$. On the other hand,
  \[ |\zeta^2_{\{a,b\}} - \zeta^1_{\{a,b\}} - \zeta^2_{\{c,d\}} + \zeta^1_{\{c,d\}}| < \frac{2\pi}{k-1}\]
by the triangle inequality. Since the term inside the absolute values is equivalent to $\theta_{\{a,b\}} - \theta_{\{c,d\}}$ (modulo $2 \pi$), it follows that $\theta_{\{a,b\}} - \theta_{\{c,d\}} \equiv 0 \pmod{2 \pi}.$ The fact that this holds for all coordinates implies that $\vec \theta$ is equivalent to a multiple of $\vec 1 \pmod{2 \pi}$.
\end{proof}

\begin{cor} \label{disjoint_regions}
  If $\delta < \frac{\pi}{2(k-1)}$, the regions $R^{\delta}_A$ and $R^{\delta}_B$ are disjoint.
\end{cor}

\begin{proof}
  Suppose, by way of contradiction, that $R^{\delta}_A$ and $R^{\delta}_B$ are not disjoint. Then there are vectors $\vec \ell^1, \vec \ell^2, \vec \zeta^1, \vec \zeta^2$ such that $\vec \ell^1 + \vec \zeta^1 \equiv \vec \ell^2 + \vec \zeta^2 \pmod{2 \pi}$ with $ \vec \ell^1 \in \Lambda_X, \vec \ell^2 \in \Lambda_0 \setminus \Lambda_X$, and $|\zeta^{i}_{\{a,b\}}| < \delta$ for $i = 1, 2$ and all choices of $a, b$.  From Lemma \ref{differences_of_lambda_not} we see that $\vec \ell^1 - \vec \ell^2$ is equivalent to a scalar multiple of $\vec 1$, which is necessarily in $\Lambda_X$ (Proposition \ref{in_lambda}). But since $\Lambda_X$ is closed under subtraction (Remark \ref{lambda_is_closed}), it cannot hold that $\vec \ell^1 - \vec \ell^2 \in \Lambda_X$, yielding a contradiction.
\end{proof}

\begin{cor} \label{lambda_tubes_disjoint}
  Suppose $\delta < \frac{\pi}{2(k-1)}$ and that we have $\vec \mu^1, \vec \mu^2 \in R^{\delta}_A$ with $\vec \mu^1 \equiv \vec \mu^2 \pmod {2 \pi}$. Let $\vec \mu^1 = \vec \ell^1 + \vec \zeta^1$ and $\vec \mu^2 = \vec \ell^2 + \vec \zeta^2$, and using the notation of Lemma \ref{lambda_structure} let $\vec \ell^1$ be defined (modulo $2\pi$) by coefficients $\gamma^1, m_i^1$ and let $\vec \ell^2$ be defined (modulo $2\pi$) by coefficients $\gamma^2, m_i^2$.  Then for all $i$, it must follow that $m_i^1 = m_i^2$.
\end{cor}

\begin{proof}
  Since $\vec \ell^1 = \vec \ell^2 + c \vec 1$ for some multiple $c$, this follows immediately from the uniqueness of the coefficients in Lemma \ref{lambda_structure}.
\end{proof}

\begin{rmk}
  The purpose of this corollary is to show that while expressions of vectors in $R_A^{\delta}$ are certainly not unique, they are unique up to the diagonal components of $\Lambda_X$, which are determined by the coefficients $m_i$.  We will eventually want to decompose $R_A^{\delta}$ into a collection of tubes, and it will be important that these tubes are disjoint, which is what is proved by this lemma.
\end{rmk}

We now discuss the full anatomy of the integral used in the Fourier inversion formula.  For convenience of notation, we define 
\[ I_{v,k}(t) = (2 \pi)^{-d} \int_{[-\pi, \pi]^d} \Phi_Y(\vec \theta)^t \ud \vec \theta \, .\]
Here, the parameter $v$ is implicitly involved in determining $d = \binom{v}{2}$, and both $v$ and $k$ are used implicitly to define the walk $Y_t$. We define the following sets to discuss the integral:
  \begin{align*}
    R_{A, \equiv}^{\delta} &= \{ \vec \mu : \vec \mu \in [-\pi, \pi)^d \textrm{ and } \vec \mu \equiv \vec \ell \pmod{2 \pi} \textrm{ for some } \vec \ell \in R_A^{\delta} \} \\
    R_{B, \equiv}^{\delta} &= \{ \vec \mu : \vec \mu \in [-\pi, \pi)^d \textrm{ and } \vec \mu \equiv \vec \ell \pmod{2 \pi} \textrm{ for some } \vec \ell \in R_B^{\delta} \} \\
    R_{C, \equiv}^{\delta} &= \{ \vec \mu : \vec \mu \in [-\pi, \pi)^d \textrm{ and } \vec \mu \equiv \vec \ell \pmod{2 \pi} \textrm{ for some } \vec \ell \in R_C^{\delta} \} \\    
  \end{align*}
Since any vector whose entries are all multiples of $2\pi$ is in $\Lambda_X$ (hence, $\Lambda_0$), and since both sets are closed under subtraction, it follows that $R_{A, \equiv}^{\delta} \subset R_A^{\delta}$, $R_{B, \equiv}^{\delta} \subset R_B^{\delta}$, and $R_{C, \equiv}^{\delta} \subset R_C^{\delta}$.
When $\delta < \frac{\pi}{2(k-1)}$, by Corollary \ref{disjoint_regions} we have
  \begin{equation} \label{integral_breakup_1}
    (2 \pi)^d I_{v,k} (t) = \int_{R_{A, \equiv}^{\delta}} \Phi_Y (\vec \theta)^t \ud \vec \theta +  \int_{R_{B, \equiv}^{\delta} \cup R_{C, \equiv}^{\delta}} \Phi_Y (\vec \theta)^t \ud \vec \theta 
  \end{equation}
which is motivated by segregating the region where $|\Phi_Y(\vec \theta)^t|$ is close to $1$ (that is, $R_A^{\delta}$) from those where it is not.  

To further analyze the integral over $R_{A, \equiv}^{\delta}$, we recall from Remark \ref{lambda_is_lines} that $\Lambda_Y (= \Lambda_X)$ consists of a disjoint union of dimension $1$ subsets of $[-\pi, \pi]^d$, all parallel to the vector $\vec 1$.  Accordingly, the region $R^{\delta}_A$ consists of a disjoint union of `tubes' surrounding lines parallel to the vector $\vec 1$.  We formalize this notion by defining the following subsets of the equivalence classes $[-\pi, \pi)^d$. Let $\vec \psi$ be a fixed vector in $\Lambda_Y^{\star}$, as defined in \eqref{def_lambda_star}:
  \begin{equation} \label{def_tube_lambda}
    T_{\vec \psi}^{\delta} = \{ \vec \mu: \vec \mu \in [-\pi, \pi)^d \textrm{ and } \vec \mu \equiv \vec \psi + \gamma \vec 1 +  \vec \zeta \pmod{2 \pi} \textrm{ where } \gamma \in [0, 2\pi) \textrm{ and } |\zeta_{\{i,j\}}| < \delta \textrm{ for all } \{i, j\} \} .
  \end{equation}
 This definition sets $T^{\delta}_{\vec \psi}$ as the `tube' in $[-\pi, \pi)^d$ that contains the vector $\vec \psi$ (modulo $2 \pi$). 
 
 From here, we can re-express $R_{A, \equiv}^{\delta}$ as a union of the $T_{\vec \psi}^{\delta}$ pieces: namely, 
 \begin{equation}  \label{tube_decomposition}
    R_{A, \equiv}^{\delta} = \bigcup_{\vec \psi \in \Lambda_Y^{\star}}T_{\vec \psi}^{\delta}.
 \end{equation}
If $\delta < \frac{\pi}{2(k-1)}$, then this union is disjoint, because if $\vec \psi^1 + \gamma^1 \vec 1 + \vec \zeta^1 \equiv \vec \psi^2 + \gamma^2 \vec 1 + \vec \zeta^2 \pmod{2 \pi}$ with $\vec \psi^a \in \Lambda_Y^{\star}$, $\gamma^a \in [0, 2 \pi)$, and $|\zeta^a_{\{i,j\}}|< \delta$, then Corollary \ref{lambda_tubes_disjoint} and the uniqueness of the coefficients $m_i$ in Lemma \ref{lambda_structure} imply that $\vec \psi^1 = \vec \psi^2$. We recall from \eqref{counting_lambda} that $|\Lambda_Y^{\star}| = (k-1)^{v-1}$.

We now use \eqref{tube_decomposition} to reconsider the integral in \eqref{integral_breakup_1}, which yields
  \begin{equation} \label{integral_breakup_2}
   (2 \pi)^d I_{v,k}(t) = \sum_{\vec \psi \in \Lambda_Y^{\star}} \int_{T_{\vec \psi}^{\delta}} \Phi_Y(\vec \theta)^t \ud \vec \theta + \int_{R_{B, \equiv}^{\delta} \cup R_{C, \equiv}^{\delta}} \Phi_Y (\vec \theta)^t \ud \vec \theta .
  \end{equation}
We note that $\vec 0 \in \Lambda_Y^{\star}$ and so we consider the nonzero vectors $\vec \psi \in \Lambda_Y^{\star}$. If  $\vec \theta = \vec \psi + \gamma \vec 1 + \vec \zeta$, then since $\Phi_Y(\gamma \vec 1) = 1$ as implied by the proof of Lemma \ref{multiset_lemma}, Proposition \ref{char_function_properties} shows that
  \begin{equation*}
    \Phi_Y(\vec \theta) = \Phi_Y(\vec \psi) \Phi_Y(\vec \zeta).
  \end{equation*}
  Hence, it follows that 
  \begin{equation*}
     \int_{T_{\vec \psi}^{\delta}} \Phi_Y(\vec \theta) \ud \vec \theta = \Phi_Y(\vec \psi) \int_{T^{\delta}_{\vec 0}} \Phi_Y(\vec \theta) \ud \vec \theta
  \end{equation*}
whence \eqref{integral_breakup_2} becomes
  \begin{equation} \label{integral_breakup_3}
    (2 \pi)^d I_{v,k}(t) = \left(\sum_{\vec \psi \in \Lambda_Y^{\star}} \Phi_Y(\vec \psi)^t\right)\int_{T_{\vec 0}^{\delta}} \Phi_Y(\vec \theta)^t \ud \vec \theta + \int_{R_{B, \equiv}^{\delta} \cup R_{C, \equiv}^{\delta}} \Phi_Y (\vec \theta)^t \ud \vec \theta.
  \end{equation}
  
Finally, we note by Proposition \ref{bibd_conditions_walk} that if $t \frac{k(k-1)}{v(v-1)} \in \Z$ but $t \frac{k}{v} \not \in \Z$, then the sum in the parentheses of \eqref{integral_breakup_3} is $0$ and we have
  \begin{equation} \label{integral_breakup_4}
    (2 \pi)^d I_{v,k}(t) = \int_{R_{B, \equiv}^{\delta} \cup R_{C, \equiv}^{\delta}} \Phi_Y (\vec \theta)^t \ud \vec \theta.
  \end{equation}
 On the other hand, if $t \frac{k(k-1)}{v(v-1)} \in \Z$ and $t \frac{k}{v} \in \Z$, then by Proposition \ref{bibd_conditions_walk}, \eqref{integral_breakup_3} becomes
  \begin{equation} \label{integral_breakup_5}
    (2 \pi)^d I_{v,k}(t) = (k-1)^{v-1} \int_{T_{\vec 0}^{\delta}} \Phi_Y(\vec \theta)^t \ud \vec \theta + \int_{R_{B, \equiv}^{\delta} \cup R_{C, \equiv}^{\delta}} \Phi_Y (\vec \theta)^t \ud \vec \theta.
  \end{equation}
Later, we will allow $t$ and $\delta$ to vary in a certain way together, so that the integral over $R_{B, \equiv}^{\delta} \cup R_{C, \equiv}^{\delta}$ approaches zero in both \eqref{integral_breakup_4} and \eqref{integral_breakup_5}.  This corresponds to the fact that a balanced incomplete block design cannot exist unless $t \frac{k}{v} \in \Z$, which is shown by \eqref{BIBD_relation1}.


\section{Bounds Far from the Maximal Set} \label{S:upper_bounds}

Having established our decomposition of the integral, we now desire to estimate the integral terms that appear in \eqref{integral_breakup_4} and \eqref{integral_breakup_5}.  The region $R_A$ is the set that is ``near'' $\Lambda_X$ and will contribute the bulk of the integral, so our goal is to provide upper bounds for the integrand on the regions $R_B^{\delta}$ and $R_C^{\delta}$ to show that their contribution is negligible when compared to that of $R_A^{\delta}$.  We begin with the integrand on the region $R_B^{\delta}$.

\begin{lemma} \label{region_b_estimate}
  Suppose $\delta < k^{-2}\binom{v}{k}^{-2} \left[ \frac{1}{6 \cdot 96^2} \left( \frac{2\pi}{k-1} \right)^4 \right]$.  Then if $\vec \mu \in R_B^{\delta}$, we have
  \[ |\Phi_X(\vec \mu)| \leq 1 - \binom{v}{k}^{-1} \left[ \frac{1}{96} \left( \frac{2 \pi}{k-1} \right)^2 \right].    \]
\end{lemma}

\begin{rmk}
  The essential point is that the bound holds when $\delta$ is sufficiently small in a manner that depends only on the preset and fixed parameters $v$ and $k$.  In the sequel, we will allow $\delta \to 0$ and the exact threshold for when the bound applies will not be of importance.  
\end{rmk}

\begin{rmk} 
  Our previous assumptions on $v$ and $k$ are that $k \geq 2$ and $v - k \geq 2$.  We notice that in the particular case where $k = 2$, the set $R_B^{\delta}$ is empty.  This is because the defining characteristic of $\Lambda_0$ simply reduces to all coordinates being congruent to one another modulo $2 \pi$; hence, taken modulo $2 \pi$ the vector is a multiple of $\vec 1$.  By Proposition \ref{in_lambda}, vectors which satisfy this condition are necessarily in $\Lambda_X$, implying that $\Lambda_X = \Lambda_0$ in this case.  Since $R_B^{\delta}$ is empty, the bound in Lemma \ref{region_b_estimate} vacuously holds in this case, so we will assume that $k \geq 3$ in the proof.
\end{rmk}

\begin{proof}[Proof of Lemma \ref{region_b_estimate}]
Let $\vec x, \vec y \in V_{v,k}$; if $\vec \ell \in \Lambda_0$, then $|Z( \vec x) \cdot \vec \ell - Z(\vec y) \cdot \vec \ell| \in \frac{2 \pi}{k-1} \Z$.  Hence, taken modulo $2 \pi$, the possible values of $|Z(\vec x) \cdot \vec \ell - Z(\vec y) \cdot \vec \ell|$ are $\{0, \frac{2\pi}{k-1}, \dots, \frac{(k-2) 2\pi}{k-1} \}$.  If $\vec \ell \notin \Lambda_X$, then by \eqref{modulo_condition} there exist $\vec x, \vec y$ so that modulo $2 \pi$, we have $|Z(\vec x) \cdot \vec \ell - Z(\vec y) \cdot \vec \ell| \neq 0 \pmod{2 \pi}$.  Therefore, for $\vec \ell \in \Lambda_0 \setminus \Lambda_X$,
  \begin{align*} |\Phi_X(\vec \ell)| &=\left| \frac{1}{\binom{v}{k}} \sum_{\vec x \in V_{v,k}} e^{i \vec \ell \cdot Z(\vec x)} \right| \\
  & \leq \frac{1}{\binom v k} \left[ \left| e^{i \vec \ell \cdot Z(\vec x)} + e^{i \vec \ell \cdot Z(\vec y)} \right| + \left| \sum_{\vec w \neq \vec x, \vec y} e^{i \vec \ell \cdot Z(\vec w)} \right| \right]
  \end{align*}
and since $|e^{i a} + e^{i b}|^2 = 2 + 2 \cos(a - b)$, we have
\begin{equation} \label{region_b_bound_1}   |\Phi_X(\vec \ell)| \leq \frac{1}{\binom v k} \left[ \sqrt{2 + 2 \cos \left(\frac{2 \pi}{k-1} \right)} + \binom v k - 2  \right].  \end{equation}
  
We note that
  \[\sqrt{x} \leq 1 + x/4 \]
and that 
  \[ \cos(x) \leq 1 - \frac{x^2}{2} + \frac{x^4}{4}\]
so substituting these into \eqref{region_b_bound_1} yields
  \begin{equation} \label{region_b_bound_2}
    |\Phi_X(\vec \ell) | \leq 1 - \frac{1}{\binom{v}{k}} \left[ \frac{ \left( \frac{2\pi}{k-1} \right)^2}{4} - \frac{\left( \frac{2\pi}{k-1} \right)^4}{48} \right].
  \end{equation}
We also note that when $k \geq 3$, 
  \[ \left( \frac{2 \pi}{k-1} \right)^4 < 11 \left( \frac{2 \pi}{k-1} \right)^2 \]
and applying this to \eqref{region_b_bound_2} gives
  \begin{equation} \label{region_b_bound_3}
     |\Phi_X(\vec \ell) | \leq 1 - \frac{1}{\binom{v}{k}} \left[ \frac{1}{48} \left( \frac{2 \pi}{k-1} \right)^2 \right].
  \end{equation}
  
Now, let $\vec \mu = \vec \ell + \vec \zeta$, where $\vec \ell \in \Lambda_0 \setminus \Lambda_X$ and $|\zeta_{\{i,j\}}| < \delta$ for all $i, j$.  Since $\Phi_X(\vec \mu) = \nk^{-1} \sum_{\vec x \in V_{v,k}} e^{i \vec \mu \cdot Z(\vec x)}$, by the triangle inequality and the fact that $|\cos(a + b) - \cos(a) | \leq |b|$, we have
  \begin{align*}
   &  | \real(\Phi_X(\vec \ell + \vec \zeta)) - \real(\Phi_X(\vec \ell))|  \\ 
   & \qquad \quad = \nk^{-1} \left| \sum_{\vec x \in V_{v,k}} \left( \cos((\vec \ell + \vec \zeta) \cdot Z(\vec x) ) - \cos(\vec \ell \cdot Z(\vec x)) \right)   \right| \\
   & \qquad \quad \leq \nk^{-1} \sum_{\vec x \in V_{v,k}} \left| \vec \zeta \cdot Z(\vec x) \right|.
  \end{align*}
We note that $|\vec \zeta \cdot Z(\vec x)| \leq \binom{k}{2} \delta < k^2 \delta$, since the vector $Z(\vec x)$ is $1$ in exactly $\binom{k}{2}$ coordinates and is $0$ elsewhere.  Since $|V_{v,k}| = \nk$, this shows that 
  \[ |\real(\Phi_X(\vec \mu)) - \real(\Phi_X(\vec \ell))| \leq k^2 \delta \]
and that in particular,
  \begin{equation}  \label{region_b_bound_4}
    |\real(\Phi_X(\vec \mu))| \leq | \real( \Phi_X(\vec \ell))| + k^2 \delta.
  \end{equation}
An identical argument with sines instead of cosines shows that
  \begin{equation}  \label{region_b_bound_5}
   |\imag(\Phi_X(\vec \mu))| \leq | \imag( \Phi_X(\vec \ell))| + k^2 \delta.
  \end{equation}
By \eqref{region_b_bound_4} and \eqref{region_b_bound_5}, we have
  \begin{align*}
    |\Phi_X(\vec \mu)|^2 & = |\real(\Phi_X(\vec \mu))|^2 + |\imag(\Phi_X(\vec \mu))|^2 \\
    & \leq |\real(\Phi_X(\vec \ell))|^2 + |\imag(\Phi_X(\vec \ell))|^2 + 4 k^2 \delta + 2 k^4 \delta^2
  \end{align*}
and since our assumptions on $\delta$ imply that $k^2 \delta < 1$, we employ the estimate
  \begin{align*}
    |\Phi_X(\vec \mu)| & \leq \sqrt{| \Phi_X(\vec \ell)|^2 + 6 k^2 \delta} \\
    & \leq |\Phi_X(\vec \ell)| + \sqrt{6 k^2 \delta}.
  \end{align*}
Putting this together with \eqref{region_b_bound_3} and our assumptions on $\delta$ gives
  \[ |\Phi_X(\vec \mu)|  \leq 1 - \nk^{-1} \left[ \frac{1}{48} \cdot \left( \frac{2 \pi}{k-1} \right)^2 \right] + \nk^{-1} \left[ \frac{1}{96} \left(\frac{2 \pi}{k-1} \right)^2 \right]  \]
as desired.
\end{proof}

Next, we seek to find a bound for the integrand on the region $R_C^{\delta}$, which will be achieved with the use of Lemma \ref{near_lambda_2}.

\begin{lemma} \label{region_c_estimate}
Suppose $\delta < 4$.  Then if $\vec \mu \in R_C^{\delta}$, we have
  \[ |\Phi_X(\vec \mu)| \leq 1 - \nk^{-1} \frac{11}{48} \left( \frac{\delta}{4} \right)^2. \]
\end{lemma}

\begin{proof}
For $x \in \R$ and $y, \epsilon_0 > 0$, we say that
  \[ |x| \bmod y < \epsilon_0 \]
if there exist $z \in \Z$ and $\epsilon \in \R$ such that $x = y z + \epsilon$ and $|\epsilon| < \epsilon_0$.  Its negation is denoted
  \[ |x| \bmod y \geq \epsilon_0 \]
and signifies that for every $z \in \Z$ and $\epsilon \in \R$, if  $x - y  z = \epsilon$, then $|\epsilon| > \epsilon_0$.  


    Suppose $\vec \mu \in R_C^{\delta}$; then there must exist a choice of $a, b, c, d$ such that $|\mu_{\{a,b\}} - \mu_{\{c,d\}}| \bmod  \frac{2\pi}{k-1} \geq \delta$. Therefore, we see by Lemma \ref{near_lambda_2} that there are vectors $\vec x, \vec y \in V_{v,k}$ for which $|Z(\vec x) \cdot \vec \mu - Z(\vec y) \cdot \vec \mu| \bmod 2 \pi \geq \delta/4$.  This condition implies that 
  \begin{equation} \label{region_c_bound_1}
    \cos(Z(\vec x) \cdot \vec \mu - Z(\vec y) \cdot \vec \mu) \leq \cos(\delta/4).
  \end{equation}
When computing $\Phi_X(\vec \mu)$, we use the same arguments and calculations that led to \eqref{region_b_bound_1} and \eqref{region_b_bound_2}, but with $\delta/4$ in place of $\frac{2 \pi}{k-1}$ inside the cosine function, to obtain
  \begin{equation*}
    |\Phi_X(\vec \mu)| \leq 1 - \nk^{-1} \left[ \frac{(\delta/4)^2}{4} - \frac{(\delta/4)^4}{48} \right].
  \end{equation*}
Then if $\delta < 4$, we have
  \[ |\Phi_X(\vec \mu)| \leq 1 - \nk^{-1} \frac{11}{48} \left( \frac{\delta}{4} \right)^2 \]
as desired.
\end{proof}

Having established our bounds on the integrands on regions $R_B^{\delta}$ and $R_C^{\delta}$ (and in particular on their subsets $R_{B, \equiv}^{\delta}$ and $R_{C, \equiv}^{\delta}$), we are now prepared to bound the corresponding integrals in \eqref{integral_breakup_4} and \eqref{integral_breakup_5}.    The previous lemmas give rise to the following upper bound on the regions of the integral that are far from $\Lambda_X$. 

\begin{prop} \label{bad_region_decay}
  When $\delta < k^{-2} \nk^{-2} \left[ \frac{1}{6 \cdot 96^2} \left( \frac{2\pi}{k-1} \right)^4 \right]$, 
  \[  \left| (2 \pi)^{-d} \int_{R_{B, \equiv}^{\delta} \cup R_{C, \equiv}^{\delta}} \Phi_Y(\vec \theta)^t \ud \vec \theta \right| < \exp \left( - \nk^{-1} \frac{11}{768} t \delta^2  \right) . \]
\end{prop}

\begin{proof}
We remark that since $|\Phi_Y(\vec \mu)| = |\Phi_X(\vec \mu)|$ as shown in the proof of Proposition \ref{sets_are_equal}, the bounds in Lemmas \ref{region_b_estimate} and \ref{region_c_estimate} apply to $|\Phi_Y(\vec \mu)|$ as well.  The assumption on $\delta$ implies that both Lemmas \ref{region_b_estimate} and \ref{region_c_estimate} apply.  Moreover, when this assumption on $\delta$ holds, it is easy to verify that the upper bound given in Lemma \ref{region_c_estimate} is larger than the upper bound given in Lemma \ref{region_b_estimate}. Since $R_{B, \equiv}^{\delta} \subset R_B^{\delta}$ and $R_{C, \equiv}^{\delta} \subset R_C^{\delta}$, putting the aforementioned estimates together yields
  \begin{align*}
    \left| (2 \pi)^{-d} \int_{R_{B, \equiv}^{\delta} \cup R_{C, \equiv}^{\delta}} \Phi_Y(\vec \theta)^t \ud \vec \theta \right| & \leq (2 \pi)^{-d}  \int_{R_{B, \equiv}^{\delta} \cup R_{C, \equiv}^{\delta}} |\Phi_Y(\vec \theta)|^t \ud \vec \theta \\
    & < \left[ 1 - \nk^{-1} \frac{11}{48} \left( \frac{\delta}{4} \right)^2 \right]^t \\
    & \leq \exp \left( - \nk^{-1} \frac{11}{768} t \delta^2 \right) \, . \qedhere
  \end{align*}  
\end{proof}
\section{Bounds Near the Maximal Set} \label{S:near_lambda}

We now seek to analyze the integrand in the region $R_{A, \equiv}^{\delta}$.  By considering \eqref{integral_breakup_5}, we see that our primary concern will be to determine bounds for the integral on the region $T_{\vec 0}^{\delta} \subset R_{A, \equiv}^{\delta}$.  We first define some combinatorial terms; for $j \in \Z^+$ with $j \leq v$, we set
  \begin{equation} \label{definition_c_j}
     C_j = \frac{\prod_{i=0}^{j-1} (k-i) }{\prod_{i=0}^{j-1} (v-i) }
  \end{equation}
and we note that if $j \leq k$, then
  \[ C_j = \frac{\binom{k}{j}}{\binom{v}{j}}  \]
whereas if $j > k$ then $C_j = 0$.  (Although the $C_j$ terms depend on both parameters $v$ and $k$, we will opt to omit this from the notation.)

We also define a $d \times d$ matrix $N$.  We regard the indices of $N$ in the same way that we regard the indices of $\R^d$; that is, its indices are sets of the form $\{a, b\}$ with $1 \leq a < b \leq v$.  Entries in the matrix $N$ will be denoted by $N_{\{a, b\}, \{c,d\}}$.  We define these entries in terms of the aforementioned combinatorial coefficients $C_j$, as follows:  
  \begin{equation} \label{matrix_coeff_def}
    N_{\{a,b\}, \{c,d\}} =
      \begin{cases}
        C_2 - C_2^2, &  |\{a, b\} \cap \{c, d\}| = 2 \\
        C_3 - C_2^2, &  |\{a, b\} \cap \{c, d\}| = 1 \\
        C_4 - C_2^2, &  |\{a, b\} \cap \{c, d\}| = 0
      \end{cases}
  \end{equation}
This makes $N$ a real, symmetric matrix.

\begin{prop} \label{N_is_singular}
  With $N$ as defined in \eqref{matrix_coeff_def} and with $k \geq 2$ and $v - k \geq 2$, we have $N \vec 1 = \vec 0$ and $\vec 1^T N = \vec 0^T$.
\end{prop}

\begin{proof}
  We first note two easily-verified computations:
  \begin{equation} \label{coefficient_combinatorics_1}
    1 + 2(v-2) + \binom{v-2}{2} = d
  \end{equation}
and 
  \begin{equation} \label{coefficient_combinatorics_2}
    C_2 + 2(v-2) C_3 + \binom{v-2}{2} C_4 =  d \cdot C_2^2 \, .
  \end{equation}

  We will show that the sum of the columns of $N$ is $\vec 0$.  For a fixed $\{a, b\}$, we consider coordinates of the form $\{c, d\}$.  Exactly one coordinate (namely, $\{a,b\}$) has $|\{a,b\} \cap \{c,d\}| = 2$, exactly $2(v-2)$ coordinates have $|\{a,b\} \cap \{c,d\}| = 1$, and exactly $\binom{v-2}{2} = \frac{(v-2)(v-3)}{2}$ coordinates have $|\{a,b\} \cap \{c,d\}| = 0$.  The claim that $N \vec 1 = \vec 0$ then amounts to showing that
  \[(C_2 - C_2^2) + (C_3 - C_2^2) \cdot 2(v-2) + (C_4 - C_2^2) \cdot \binom{v-2}{2} = 0   \]
which follows immediately from \eqref{coefficient_combinatorics_1} and \eqref{coefficient_combinatorics_2}.  The claim $\vec 1^T N = \vec 0^T$ then follows from the symmetry of $N$.
\end{proof}

To motivate the construction of the matrix $N$, we let $\vec \xi$ be a randomly-selected element of $V_{v,k} \subset \R^v$ and we recall that $Z(\vec \xi) = (\xi_1 \xi_2, \xi_1 \xi_3, \dots, \xi_{v-1} \xi_v)$.  We also recall that the random walk $Y_t$ has increments of the form $Z(\vec \xi) - C_2 \vec 1$ where $\xi$ is chosen randomly and uniformly from the elements in $V_{v,k}$.  For $\vec \mu \in [-\pi, \pi]^d$, we will be interested in computing and estimating quantities of the form
  \begin{equation}  \label{expectation_p}
    \E  \left[ \left(\vec \mu \cdot (Z(\vec \xi) - C_2 \vec 1) \right)^p \right]
  \end{equation}
for $p = 1, 2, 3, 4$.  The purpose of the matrix $N$ is the following proposition:

\begin{prop} \label{second_moment}
  Let $\vec \mu \in [-\pi, \pi)^d$.  Then
    \[ \E  \left[ \left(\vec \mu \cdot (Z(\vec \xi) - C_2 \vec 1) \right)^2 \right] = \vec \mu^T N \vec \mu \, .  \]
\end{prop}

\begin{proof}
  The left term is
    \begin{align*}
      \E  \left[ \left(\vec \mu \cdot (Z(\vec \xi) - C_2 \vec 1) \right)^2 \right]
      & = \E \left[  \left( \sum_{\{a,b\}} \mu_{\{a,b\}} (\xi_a \xi_b - C_2)  \right)^2  \right] \\
      & = \sum_{\{a,b\}, \{c,d\}} \mu_{\{a,b\}} \mu_{\{c,d\}} \E \left[  (\xi_a \xi_b - C_2)(\xi_c \xi_d - C_2)  \right] 
    \end{align*} 
where the last sum is taken over all ordered pairs of coordinate sets.  To prove the result, we must show that this quadratic form agrees with the entries of $N$; that is, that $\E[(\xi_a \xi_b - C_2)(\xi_c \xi_d - C_2)]$ is given by the coefficients of $N$ in \eqref{matrix_coeff_def}.

We first consider the terms in the sum where $|\{a,b\} \cap \{c,d\}| = 2$; that is, $\{c, d\} = \{a, b\}$.  Here, 
  \begin{equation} \label{abab_moment_1}
    \E[(\xi_a \xi_b - C_2)(\xi_a \xi_b - C_2)]  = \E[\xi_a \xi_b - 2 C_2 \xi_a \xi_b + C_2^2 ]
  \end{equation}
since all vectors in $V_{v,k}$ have entries that are either $0$ or $1$.  The product $\xi_a \xi_b$ will be 1 if $\xi_a = 1$ and $\xi_b = 1$; otherwise, it will be $0$.  Of the $\nk$ vectors in $V_{v,k}$, there are $\binom{v-2}{k-2}$ vectors which have $\xi_a = 1$ and $\xi_b = 1$, corresponding to the ways to select the locations for the remaining $k-2$ $1$'s from the remaining $v-2$ possible positions.  Hence, the probability that $\xi_a \xi_b$ is $1$ is $\binom{v-2}{k-2}/\binom{v}{k} = C_2$, from which it follows that
  \begin{equation} \label{first_moment}
    \E[\xi_a \xi_b] = C_2 \, .
  \end{equation}
Substituting this into \eqref{abab_moment_1} gives
  \[
    \E[(\xi_a \xi_b - C_2) (\xi_a \xi_b - C_2)] = C_2 - C_2^2 
  \]
which agrees with the corresponding coefficient of $N$.

Next, we consider the terms in the sum where $|\{a, b\} \cap \{c, d\}| = 1$ by considering an index pair of the form $\{a, b\}, \{a, c\}$.  In this case,
  \begin{equation} \label{abac_moment_1}
    \E[(\xi_a \xi_b - C_2)(\xi_a \xi_c - C_2)]  = \E[\xi_a \xi_b \xi_c - C_2 \xi_a \xi_b - C_2 \xi_a \xi_c + C_2^2 ] \, .
  \end{equation}
By analyzing the first term in a fashion similar to our discussion of \eqref{first_moment}, we see that $\E[\xi_a \xi_b \xi_c] = \binom{v-3}{k-3}/ \nk = C_3$.  Using this and \eqref{first_moment} in \eqref{abac_moment_1} shows that
  \[
    \E[(\xi_a \xi_b - C_2) (\xi_a \xi_c - C_2)] = C_3 - C_2^2 
  \]
which again agrees with the corresponding coefficient of $N$.  

Finally, we consider the case where $|\{a, b\} \cap \{c, d\} = 0|$; that is, $a, b, c, d$ are all distinct.  Here,
  \begin{equation} \label{abcd_moment_1}
    \E[(\xi_a \xi_b - C_2)(\xi_c \xi_d - C_2)]  = \E[\xi_a \xi_b \xi_c \xi_d - C_2 \xi_a \xi_b - C_2 \xi_c \xi_d + C_2^2 ]
  \end{equation}
and as before, the expectation of the first term is $\binom{v-4}{k-4} / \nk = C_4$, whence \eqref{abcd_moment_1} becomes
  \[
    \E[(\xi_a \xi_b - C_2) (\xi_c \xi_d - C_2)] = C_4 - C_2^2 
  \]
which also agrees with the corresponding entry of $N$.
\end{proof}

\begin{rmk} \label{mean_zero}
  The process $Y_t$ was defined as being the process $X_t$ with a drift correction, which corresponds to the calculation in \eqref{first_moment}.  That calculation shows that the term in \eqref{expectation_p} is $0$ when $p = 1$.  We have now calculated the term when $p = 2$; in the following Lemma, we will estimate (rather than compute) the terms with $p = 3$ and $p = 4$.
\end{rmk}

\begin{lemma} \label{good_region_bounds}
  Let $\delta > 0$.  Then there is a function $\varepsilon_1: T_{\vec 0}^{\delta} \to \R$ such that for all $\vec \mu \in T_{\vec 0}^{\delta}$, we have
  \begin{equation} \label{phi_real_bound}
    \real(\Phi_Y(\vec \mu)) = e^{-\frac{1}{2} \vec \mu^T N \vec \mu}( 1 + \varepsilon_1(\vec \mu))
  \end{equation}
and $|\varepsilon_1(\vec \mu)| < \frac{1}{6} (d \delta)^4 e^{\frac{1}{2} d^2 \delta^2}$.  Moreover,
  \begin{equation} \label{phi_imag_bound}
    |\imag(\Phi_Y(\vec \mu))| \leq \frac{(d \delta)^3}{6} \, .
  \end{equation}
Further, if $d \delta < 1$, then for $\vec \mu \in T_{\vec 0}^{\delta}$ we have
  \begin{equation} \label{phi_real_lower_bound}
    \real(\Phi_Y(\vec \mu)) \geq \frac{1}{3}.
  \end{equation}  
\end{lemma}

\begin{proof}
For this proof, we will mimic the proof of Lemma 3.1 in \cite{wdl_levin}. Since $\vec \mu \in T_{\vec 0}^{\delta}$, we have that $\vec \mu \equiv \gamma \vec 1 + \vec \zeta \pmod{2 \pi}$, where $\gamma \in [0, 2\pi)$ and $|\zeta_{\{i,j\}}| < \delta$ for all $\{i, j\}$. Because we are concerned only with bounds on $|\Phi_Y(\vec \mu)|$ and $\Phi_Y$ is $2 \pi$-periodic, we will assume without loss of generality that
  \begin{equation} \label{tube_0_decomp}
    \vec \mu = \gamma \vec 1 + \vec \zeta
  \end{equation}
 where $|\zeta_{\{i,j\}}| < \delta$ for all $\{i,j\}$.  
We begin with the remainder bounds on Taylor polynomials for $e^z$.  If $a \geq 0$ and $b$ is real, we have
  \begin{align}
    \left| e^{-a} - \sum_{s=0}^j \frac{(-a)^s}{s!} \right| & \leq \min \left\{ \frac{2|a|^j}{j!}, \frac{|a|^{j+1}}{(j+1)!} \right\}, \label{taylor_bound_1}\\
    \left| e^{ib} - \sum_{s=0}^j \frac{(ib)^s}{s!} \right| & \leq \min \left \{ \frac{2|b|^j}{j!}, \frac{|b|^{j+1}}{(j+1)!} \right\}. \label{taylor_bound_2}
  \end{align}
  For a reference, one can find \eqref{taylor_bound_1} as \cite{billingsley}*{equation 26.4}; \eqref{taylor_bound_2} is proved similarly.  Using \eqref{taylor_bound_1} with $j = 1$ shows that 
  \begin{equation} \label{taylor_mean_1}
    \left| e^{-\frac{1}{2} \vec \mu^T N \vec \mu} - \left(1 - \frac{1}{2} \vec \mu^T N \vec \mu \right) \right| \leq  \frac{1}{8} (\vec \mu^T N \vec \mu)^2.
  \end{equation}
By \eqref{tube_0_decomp} and Proposition \ref{N_is_singular}, we note that
  \begin{align*}
    \vec \mu^T N \vec \mu & = (\gamma \vec 1^T + \vec \zeta^T) N (\gamma \vec 1 + \vec \zeta) \\
    & = \vec \zeta^T N \vec \zeta.
  \end{align*}
We note from the triangle inequality that
  \[ | \vec \zeta^T N \vec \zeta| \leq \sum_{\{a,b\}, \{c,d\}} |\zeta_{\{a,b\}} \zeta_{\{c,d\}} N_{\{a,b\},\{c,d\}}|   \]
and we observe that all coefficients of $N$ have absolute value at most $1$ since $0 \leq C_j < 1$ for $j = 2, 3, 4$.  Since the components of $\vec \zeta$ are bounded by $\delta$, it follows that 
  \begin{equation} \label{muNmu_bound}
    |\vec \mu^T N \vec \mu| < \sum_{\{a,b\}, \{c,d\}} \delta^2 = d^2 \delta^2\, .
  \end{equation}    
Using this in conjunction with \eqref{taylor_mean_1} establishes that
  \begin{equation} \label{taylor_mean_2}
     \left| e^{-\frac{1}{2} \vec \mu^T N \vec \mu} - \left(1 - \frac{1}{2} \vec \mu^T N \vec \mu \right) \right| \leq  \frac{1}{8} d^4 \delta^4.
  \end{equation}
  
Next, let $\vec y$ be any vector in $V_{v,k}$.  For convenience of notation, we set $W(\vec y) = Z(\vec y) - C_2 \vec 1$.  Using \eqref{taylor_bound_2} with $j = 3$ implies that
  \begin{align*}
    &\left| e^{i \vec \mu \cdot W(\vec y)} - \left[ 1 + i \vec \mu \cdot W(\vec y) - \frac{ (\vec \mu \cdot W( \vec y))^2}{2} - \frac{i  (\vec \mu \cdot W(\vec y))^3}{6} \right] \right|  \leq \frac{1}{24} (\vec \mu \cdot W(\vec y))^4.
  \end{align*}
Using this with the fact that $|\real(z)| < |z|$ for any $z \in \C$, we see that
  \begin{equation}
     \left| \real(e^{i \vec \mu \cdot W(\vec y)}) - \left[1 -  \frac{1}{2} (\vec \mu \cdot W(\vec y))^2 \right] \right| \leq \frac{1}{24} (\vec \mu \cdot W(\vec y))^4. \label{taylor_mean_3a}
  \end{equation}
We now let $\vec y$ be a random, uniformly-chosen element of $V_{v,k}$.  From \eqref{taylor_mean_3a}, we see that
  \begin{align}
   & \left| \E \left[ \real(e^{i \vec \mu \cdot W(\vec y)}) \right]  - \E \left[1 - \frac{1}{2} (\vec \mu \cdot W(y))^2 \right] \right| \nonumber \\
   & \qquad \quad  \leq \E \left| \real(e^{i \vec \mu \cdot W(\vec y)}) - \left[ 1 - \frac{1}{2}(\vec \mu \cdot W(\vec y))^2 \right]  \right| \nonumber \\
   & \qquad \quad \leq \frac{1}{24} \E [(\vec \mu \cdot W(\vec y))^4]. \label{taylor_mean_4a}
  \end{align}
Since $\real$ is linear, we have $\E[\real(e^{i \vec \mu \cdot W(\vec y)})] = \real(\Phi_Y(\vec \mu))$.  Hence, \eqref{taylor_mean_4a} and Proposition \ref{second_moment} combine to yield
  \begin{equation} \label{taylor_mean_5a}
    \left| \real(\Phi_Y(\vec \mu)) - \left[ 1 - \frac{1}{2} \vec \mu^T N \vec \mu \right] \right| \leq \frac{1}{24} \E[(\vec \mu \cdot W (\vec y))^4].
  \end{equation}

To obtain a preliminary bound on $\imag(\Phi_Y(\vec \mu))$, we set $j = 2$ in \eqref{taylor_bound_2} to obtain
  \begin{equation*}
    \left| e^{i \vec \mu \cdot W(\vec y)} - \left[ 1 + i \vec \mu \cdot W(\vec y) - \frac{1}{2} (\vec \mu \cdot W( \vec y))^2 \right] \right|
    \leq \frac{1}{6} |\vec \mu \cdot W(\vec y)|^3
  \end{equation*}
and since $|\imag(z)| < |z|$, we have
  \begin{equation*}
    \left| \imag(e^{i \vec \mu \cdot W(\vec y)}) - \vec \mu \cdot W(\vec y) \right| \leq \frac{1}{6} |\vec \mu \cdot W(\vec y)|^3.
  \end{equation*}
Using the same argument as for the real part, we see that if $\vec y$ is a random, uniformly-chosen element of $V_{v,k}$,
  \begin{equation*} 
    \left| \imag(\Phi_Y(\vec \mu)) - \E[\vec \mu \cdot W(\vec y)] \right| \leq \frac{1}{6} \E[|\vec \mu \cdot W(\vec y)|^3]
  \end{equation*}
and by Remark \ref{mean_zero} we have $\E[\vec \mu \cdot W(\vec y)] = 0$, so it follows that
  \begin{equation} \label{taylor_mean_7a}
    \left| \imag(\Phi_Y(\vec \mu)) \right| \leq \frac{1}{6} \E[|\vec \mu \cdot W(\vec y)|^3].
  \end{equation}

To prove \eqref{phi_real_bound} and \eqref{phi_imag_bound}, we need to bound the expectations in \eqref{taylor_mean_5a} and \eqref{taylor_mean_7a}.  For any $\vec y \in V_{v,k}$, we have $\vec 1 \cdot Z(\vec y) = \binom{k}{2}$, and $\vec 1 \cdot C_2 \vec 1 = \frac{k(k-1)}{v(v-1)} \frac{v(v-1)}{2} = \binom{k}{2}$; hence, using $\vec \mu = \gamma \vec 1 + \vec \zeta$ from \eqref{tube_0_decomp} shows that
  \[ \vec \mu \cdot W( \vec y) = (\gamma \vec 1 + \vec \zeta) \cdot (Z(\vec y) - C_2 \vec 1) = \vec \zeta \cdot W(y)\, .\]  For any $\vec y \in V_{v,k}$, the components of $W(\vec y)$ all have absolute value at most $1$.  Since the components of $\vec \zeta$ have absolute value at most $\delta$, by the triangle inequality we have 
  \begin{equation} \label{taylor_mean_8a}
    | \vec \mu \cdot W(\vec y)| \leq \sum_{\{a,b\}} | \zeta_{\{a,b\}}| \leq d \delta. 
  \end{equation}
Combining \eqref{taylor_mean_8a} with \eqref{taylor_mean_7a} yields \eqref{phi_imag_bound}.  Likewise, using \eqref{taylor_mean_8a} with \eqref{taylor_mean_5a} shows that
  \begin{equation*}
    \left| \real(\Phi_Y(\vec \mu)) - \left[ 1 - \frac{1}{2} \vec \mu^T N \vec \mu \right] \right| \leq \frac{(d \delta)^4}{24} 
  \end{equation*}
and combining this with \eqref{taylor_mean_2} via the triangle inequality gives
  \begin{equation*}
    \left| \real(\Phi_Y(\vec \mu)) - e^{-\frac{1}{2} \vec \mu^T N \vec \mu} \right| \leq \frac{(d \delta)^4}{6}.
  \end{equation*}Dividing both sides by $e^{-\frac{1}{2} \vec \mu^T N \vec \mu}$ yields
  \begin{equation*}
    \left| \frac{\real(\Phi_Y(\vec \mu))}{e^{-\frac{1}{2} \vec \mu^T N \vec \mu}} - 1 \right| \leq  \frac{1}{6} (d \delta)^4 e^{\frac{1}{2} \vec \mu^T N \vec \mu}
  \end{equation*}
and by \eqref{muNmu_bound}, we see that
  \begin{equation*}
    \left| \frac{\real(\Phi_Y(\vec \mu))}{e^{-\frac{1}{2} \vec \mu^T N \vec \mu}} - 1 \right| \leq \frac{1}{6} (d \delta)^4 e^{\frac{1}{2} d^2 \delta^2} \, .
  \end{equation*}
Therefore, we have
  \begin{equation*}
    \real(\Phi_Y(\vec \mu)) = e^{-\frac{1}{2} \vec \mu^T N \vec \mu} \left[ \frac{\real(\Phi_Y(\vec \mu))}{e^{-\frac{1}{2} \vec \mu^T N \vec \mu}} \right] = e^{-\frac{1}{2} \vec \mu^T N \vec \mu}(1 + \varepsilon_1(\vec \mu))
  \end{equation*}
where
  \begin{equation*}
    |\varepsilon_1(\vec \mu)| \leq \frac{1}{6} (d \delta)^4 e^{\frac{1}{2} d^2 \delta^2}
  \end{equation*}
which establishes \eqref{phi_real_bound}.

Finally, to establish \eqref{phi_real_lower_bound}, we note from \eqref{phi_real_bound} that
  \[ \real(\Phi_Y(\vec \mu)) \geq e^{-\frac{1}{2} \vec \mu^T N \vec \mu} \left( 1 - \frac{1}{6} (d \delta)^4 e^{\frac{1}{2} (d \delta)^2} \right) \]
and by \eqref{muNmu_bound} and the assumption that $(d \delta) < 1$, we see that
  \begin{align*}
    \real(\Phi_Y(\vec \mu)) & \geq \frac{1 - \frac{1}{6}(d \delta)^4 e^{\frac{1}{2} (d \delta)^2}}{e^{\frac{1}{2}(d \delta)^2}} \geq 1/3
  \end{align*}
as desired.
\end{proof}
\section{The Submatrix Determinant} \label{S:matrix_determinant}

We now reconsider the $d \times d$ matrix $N$ as defined in \eqref{matrix_coeff_def}.  As implied by Proposition \ref{N_is_singular} this matrix is singular.  Our primary concern in the upcoming calculations will not be $N$, but its $(d -1) \times (d -1)$ principal submatrix obtained by removing the row and column with index $\{v-1, v\}$.  We will denote this submatrix by $M$.  We will need to discuss the corresponding subspace $\R^{d-1} \subset \R^d$, so we specify that if our enumeration of the coordinates of $\R^d$ is
  \[ \{1,2\}, \{1,3\}, \dots, \{v-2, v\}, \{v-1, v\} \]
then the coordinates of $\R^{d-1}$ are enumerated as
  \[ \{1,2\}, \{1,3\}, \dots, \{v-2, v\} \]
to correspond to our definition of $M$.

\begin{lemma} \label{obnoxious_reparametrization}
With $M, N$ as previously defined, 
    \[  2 \pi \int_{[-\delta, \delta]^{d-1}} e^{-\frac{t}{2} \vec \mu^T M \vec \mu} \ud \vec \mu 
   \leq \int_{T_{\vec 0}^{\delta}} e^{-\frac{t}{2} \vec \theta^T N \vec \theta} \ud \vec \theta 
   \leq 2 \pi \int_{[-2\delta, 2\delta]^{d-1}} e^{-\frac{t}{2} \vec \mu^T M \vec \mu} \ud \vec \mu. \]
\end{lemma}

\begin{proof}
   We begin by reparametrizing the middle integral.  We define a region better suited for the upcoming reparametrization:
   \[ S_{\vec 0}^{\delta} = \{\vec \mu: \vec \mu \in [-\pi, \pi)^d \textrm{ and } \vec \mu \equiv \gamma \vec 1 + \vec \zeta \pmod{2 \pi} \textrm{ with } \gamma \in [0, 2 \pi), |\zeta_{\{i, j\}}| < \delta \textrm{ for all } i, j, \textrm{ and } \zeta_{\{v-1, v\}} = 0\}. \]
We note from \eqref{def_tube_lambda} that $S_{\vec 0}^{\delta} \subset T_{\vec 0}^{\delta}$ is clear.  Since $\gamma \vec 1 + \vec \zeta \equiv (\gamma + \zeta_{\{v-1, v\}}) \vec 1 + (\vec \zeta - \zeta_{\{v-1,v\}} \vec 1),$ the triangle inequality shows that $T_{\vec 0}^{\delta} \subset S_{\vec 0}^{2 \delta}$. Therefore, we have
 \begin{equation} \label{reparametrization}
   \int_{S_{\vec 0}^{\delta}} e^{-\frac{t}{2} \vec \theta^T N \vec \theta} \ud \vec \theta 
   \leq \int_{T_{\vec 0}^{\delta}} e^{-\frac{t}{2} \vec \theta^T N \vec \theta} \ud \vec \theta 
   \leq \int_{S_{\vec 0}^{2\delta}} e^{-\frac{t}{2} \vec \theta^T N \vec \theta} \ud \vec \theta .
 \end{equation}
 
To reparametrize the integral, we define a function $g: \R^d \to \R^d$ such that
  \begin{align*}
    g(\vec \nu)_{\{1,2\}} & = \nu_{\{1,2\}} + \nu_{\{v-1, v\}} \\
    g(\vec \nu)_{\{1,3\}} & = \nu_{\{1,3\}} + \nu_{\{v-1, v\}} \\
    & \, \, \,  \vdots \\
    g(\vec \nu)_{\{v-2, v\}} & = \nu_{\{v-2, v\}} + \nu_{\{v-1, v\}} \\
    g(\vec \nu)_{\{v-1, v\}} & = \nu_{\{v-1, v\}} .    
  \end{align*} 
It is easy to see that the Jacobian determinant of this transformation is $1$, and that
  \[g \left( [-\delta, \delta]^{d-1} \times [0, 2 \pi) \right) = S_{\vec 0}^{\delta} \, .\]
For convenience of notation, we write $\vec \nu^0 = (\nu_{\{1,2\}}, \dots, \nu_{\{v-2, v\}}, 0)^T$ and we set $\vec \theta= g(\vec \nu)$, so that $\vec \theta = \vec \nu^0 + \nu_{\{v-1, v\}} \vec 1$.  From Proposition \ref{N_is_singular}, we see that
  \begin{align*} \vec \theta^T N \vec \theta & = (\vec \nu^0 + \nu_{\{v-1, v\}} \vec 1)^T N (\vec \nu^0 + \nu_{\{v-1, v\}} \vec 1) \\
  & = (\vec \nu^0)^T N \vec \nu^0.
  \end{align*}
By applying the change of variables formula to the integral, we obtain
   \[ \int_{S_{\vec 0}^{\delta}} e^{-\frac{t}{2} \vec \theta^T N \vec \theta} \ud \vec \theta = \int_{0}^{2 \pi} \int_{-\delta}^{\delta} \dots \int_{-\delta}^{\delta} e^{-\frac{t}{2} (\vec \nu^0)^T N \vec \nu^0} \ud \nu_{\{1,2\}} \dots \ud \nu_{\{v-2, v\}} \, \ud \nu_{\{v-1, v\}} \]
and since the rightmost integrand no longer depends on $\nu_{\{v-1, v\}}$, we can integrate that variable to get
   \begin{equation}
     \int_{S_{\vec 0}^{\delta}} e^{-\frac{t}{2} \vec \theta^T N \vec \theta} \ud \vec \theta = 2 \pi \int_{-\delta}^{\delta} \dots \int_{-\delta}^{\delta} e^{-\frac{t}{2} (\vec \nu^0)^T N \vec \nu^0} \ud \nu_{\{1,2\}} \dots \ud \nu_{\{v-2, v\}} \, .  \label{changing_variables_again}
   \end{equation}

Next, we let $h: \R^d \to \R^{d-1}$ be the projection onto the first $d-1$ coordinates, and we set $\vec \mu = h(\vec \nu)$. (We introduce this notation only so that we have a convenient way to distinguish between vectors in $\R^d$ and in $\R^{d-1}$.)  Since the $\{v-1, v\}$ component of $\vec \nu^0$ is $0$, we have $(\vec \nu^0)^T N \vec \nu^0 = \vec \mu^T M \vec \mu$. Applying the change of variables formula to \eqref{changing_variables_again} then yields
   \[ \int_{S_{\vec 0}^{\delta}} e^{-\frac{t}{2} \vec \theta^T N \vec \theta} \ud \vec \theta = 2 \pi \int_{[-\delta, \delta]^{d-1}} e^{-\frac{t}{2} \vec \mu^T M \vec \mu} \ud \vec \mu \, . \]
Using this on the left and right of \eqref{reparametrization} completes the proof.
\end{proof}

Our strategy for estimating the integral in \eqref{integral_breakup_5} is as follows: we will use Proposition \ref{bad_region_decay} to show that the second integral term vanishes, and we will use Lemma \ref{good_region_bounds} to exchange the first integral term for a Gaussian integral involving $e^{-\frac{1}{2} \vec \theta^T N \vec \theta}$. Then, using Lemma \ref{obnoxious_reparametrization}, we will reparametrize this integral as one involving $e^{-\frac{1}{2} \vec \mu^T M \vec \mu}$. If we can establish that $M$ is positive definite and compute its determinant, then the remainder of the estimation is straightforward. We will first compute the determinant:

\begin{lemma} \label{determinant_calc}
  The $(d-1) \times (d-1)$ matrix $M$ has 
    \begin{equation} \label{det_M}
      \det(M) = 2 (k-1)^{d + v - 2} \left( \frac{v-k-1}{v-3} \right)^{d-v} \left( \frac{k(v-k)}{v-2} \right)^{d-1} \left( \frac{1}{v(v-1)} \right)^d .
    \end{equation}
\end{lemma}

The proof of this calculation is quite long and tedious, so we delay it until the end of this section. After obtaining this determinant, completing the strategy outlined above is not difficult.

\begin{cor} \label{M_positive_definite}
  The matrix $M$ is positive definite.
\end{cor}

\begin{proof}
  The quadratic form associated to matrix $N$ is positive semidefinite, since it corresponds to a nonnegative expectation in Propositon \ref{second_moment}. Hence, the eigenvalues of $N$ are all nonnegative.  By Cauchy's interlace theorem (see, for example, \cite{hwang}), the eigenvalues of $M$ are also all nonnegative.  But by examining their product, $\det(M)$, we note that this product is nonzero as long as $k \geq 2$ and $v - k \geq 2$.  This means that each eigenvalue is strictly positive and that $M$ is therefore positive definite.
\end{proof}

Since $M$ is positive definite, there is a unique symmetric, positive definite matrix $P$ such that $P^2 = M$.  In an upcoming integral computation, we will need to understand the set
  \[ P [-\delta, \delta]^{d-1} = \left\{ P \vec \mu: \vec \mu \in \R^{d-1} \textrm{ and } |\mu_{\{i, j\}}| < \delta \textrm{ for all } i, j    \right\}.  \]
Rather than actually computing this set, it will suffice for us to bound it.

\begin{prop} \label{scaling_sets}
  There exist positive constants $D_1, D_2$ which depend only on $v$ and $k$ such that for all $\delta > 0$,
   \[ [- D_1 \delta, D_1 \delta]^{d-1} \subset P[-\delta, \delta]^{d-1} \subset [-D_2 \delta, D_2 \delta]^{d-1}. \]
\end{prop}

\begin{proof}
  Since $P$ is positive definite, the linear transformation corresponding to $P$ maps the box $[-1, 1]^{d-1}$ to some nondegenerate subset of $\R^{d-1}$.  Therefore, there are constants $D_1$ and $D_2$ such that
  \[ [-D_1, D_1]^{d-1} \subset P[-1, 1]^{d-1} \subset [-D_2, D_2]^{d-1}. \]
These constants depend on the matrix $P$, which is defined in terms of the matrix $M$, which depends only on the constants $v$ and $k$.  We scale these sets by a factor of $\delta$ and exploit the linearity of the transformation associated to matrix $P$ to obtain
  \[[-D_1\delta, D_1 \delta]^{d-1} \subset P[-\delta, \delta]^{d-1} \subset [-D_2 \delta, D_2 \delta]^{d-1} \]
as desired.
\end{proof}

\begin{rmk}
  The salient detail of Proposition \ref{scaling_sets} is that $D_1$ and $D_2$ do not depend on $\delta$. This proposition will be needed when employing the aforementioned Gaussian integral techniques.
\end{rmk}

The remainder of this section is dedicated to computing $\det(M)$.  The first step toward this goal is finding a convenient expression of $N$ in terms of elementary matrices.  We remark here that at several points in the upcoming calculations, we will refer to $1 \times 1$ matrices, to their entries, and to their determinants interchangeably.  

Fix $n \in \N$.  We will denote the $n \times n$ identity matrix by $I_n$.  We will define $\vec x_n$ to be the vector in $\R^n$ with all entries $1$; i.e.
  \begin{equation} \label{def_x_r}
    \vec x_n = (1, \dots, 1)^T.
  \end{equation}
We will also define $y_n \in \R^n$ to be the vector with the last two entries $1$ and all other entries $0$; i.e.
  \begin{equation} \label{def_y_r}
    \vec y_n = (0, \dots, 0, 1, 1)^T.
  \end{equation}
We collect some useful computations involving these vectors:
  \begin{align}
    \vec x_n \vec x_n^T & = \label{xr_xrt} \left[ \begin{array}{ccc}  
      1& \dots &1 \\ 
      \vdots & \ddots & \vdots \\ 
      1 & \dots & 1
      \end{array} \right] \\ 
    \label{xrt_xr} \vec x_n^T \vec x_n & = n \\
    \label{yr_yrt} \vec y_n \vec y_n^T & = \left[ \begin{array}{ccccc}
      0 & \dots & 0 & 0 & 0 \\
      \vdots & \ddots & \vdots & \vdots & \vdots \\
      0 & \dots & 0 & 0 & 0 \\
      0 & \dots & 0 & 1 & 1 \\
      0 & \dots & 0 & 1 & 1
      \end{array} \right] \\
    \label{yrt_yr} \vec x_n^T \vec y_n = \vec y_n^T \vec x_n = \vec y_n^T \vec y_n & = 2
  \end{align}

We recall from Definition \ref{defn_alpha_beta}  that $\vec \beta^a$ is defined by $\beta^a_{\{i,j\}} = 1$ if $i = a$ or $j = a$ and $\beta^a_{\{i,j\}} = 0$ otherwise.  We let $\vec \chi^a$ be the vector obtained by truncating the $\{v-1, v\}$ coordinate from $\vec \beta^a$, so that $\vec \chi^a \in \R^{d-1}$.  We define a $(d - 1) \times v$ matrix $Q$ by
  \begin{equation} \label{def_Q}
    Q = \left[ \begin{array}{cccc} \vec \chi^1 & \vec \chi^2 & \dots & \vec \chi^v \end{array} \right].
  \end{equation}
For example, if $v = 5$, then
  \[
    Q = 
     \left[ \begin{array}{ccccc} 
      1 & 1 & 0 & 0 & 0 \\ 
      1 & 0 & 1 & 0 & 0 \\
      1 & 0 & 0 & 1 & 0 \\
      1 & 0 & 0 & 0 & 1 \\
      0 & 1 & 1 & 0 & 0 \\
      0 & 1 & 0 & 1 & 0 \\
      0 & 1 & 0 & 0 & 1 \\
      0 & 0 & 1 & 1 & 0 \\
      0 & 0 & 1 & 0 & 1
    \end{array} \right]
   \quad \begin{array}{c} \{1,2\} \\ \{1,3\} \\ \{1,4\} \\ \{1,5\} \\ \{2,3\} \\ \{2,4\} \\ \{2,5\} \\ \{3,4\} \\ \{3,5\} \end{array}
  \]
where the labels to the right denote the standard coordinate enumeration of $\R^{d-1}$.  We note that of all the vectors $\vec \beta^a$, the only ones that had a (now removed) 1 in the $\{v-1, v\}$ coordinate are $\vec \beta^{v-1}$ and $\vec \beta^v$.  

The primary importance of the matrix $Q$ is the computation of the $(d -1) \times (d-1)$ matrix $Q Q^T$, which can be found by examining the inner products of rows $\{a,b\}$ and $\{c,d\}$ of $Q$:
  \begin{equation} \label{Q_QT}
    (Q Q^T)_{\{a,b\},\{c,d\}} = \begin{cases}
      2, & |\{a,b\} \cap \{c,d\}| = 2 \\
      1, & |\{a,b\} \cap \{c,d\}| = 1 \\
      0, & |\{a,b\} \cap \{c,d\}| = 0 .
    \end{cases}
  \end{equation}
Comparing this computation with \eqref{matrix_coeff_def} sheds light on why $Q$ is a useful matrix.  We will also need to consider the $v \times v$ matrix $Q Q^T$, which can be expressed as
  \begin{equation} \label{QT_Q}
    Q^T Q = (v-2) I_v + \vec x_v \vec x_v^T - \vec y_v \vec y_v^T .
  \end{equation}
To see this, we consider the inner products of columns of the matrix $Q$.  The inner product of any column with itself is the number of 1's in that column, which is $v-1$ for all but the last two columns and is $v-2$ for the last two columns; these agree with the diagonal entries of the sum in \eqref{QT_Q}.  Similarly, the inner product of distinct columns $i$ and $j$ is $1$, corresponding to the $1$ found in the $\{i,j\}$ row of each column.  The exception is if $i = v-1$ and $j = v$ (or vice versa), where the inner product is $0$. These entries are also given by the sum in \eqref{QT_Q}.

We also make note of the following computation, to be used when computing $\det(M)$:
  \begin{equation} \label{xT_Q}
    \vec x_{d-1}^T Q = (v-1) \vec x_v^T - \vec y_v^T .
  \end{equation}
This follows from fact the every column in $Q$ has $v-1$ entries equal to $1$, except for the last two, which have only $v-2$ entries equal to $1$.

We are ready to express our matrix $M$ of interest in terms of these constituent parts:
\begin{prop} \label{decomposing_M}
  With matrices $M$, $I_{d-1}$, $\vec x_{d-1}$, $Q$, and coefficients $C_i$ as previously defined, and with $a_1 = C_2 - 2C_3 + C_4$, $a_2 = C_4 - C_2^2$, and $a_3 = C_3 - C_4$,
    \begin{equation} \label{M_decomposed_eq}
      M = a_1 I_{d-1} + a_2 \vec x_{d-1} \vec x_{d-1}^T + a_3 Q Q^T .
    \end{equation}
\end{prop}
\begin{proof}
  Let $R = a_1 I_{d-1} + a_2 \vec x_{d-1} \vec x_{d-1}^T + a_3 Q Q^T$.  We will verify that these entries of $R$ agree with the entries in \eqref{matrix_coeff_def} by using \eqref{xr_xrt} and \eqref{Q_QT}.  A coordinate pair of the form $(\{a,b\},\{a,b\})$ (i.e. one on the diagonal of $R$) receives a contribution from all three parts of the sum in \eqref{M_decomposed_eq}:
  \begin{align*}
    R_{\{a,b\},\{a,b\}} & = a_1 + a_2 + 2 a_3 \\
    & = C_2 - C_2^2.
  \end{align*}
A coordinate pair of the form $(\{a,b\},\{a,c\})$ (i.e. exactly one shared component) does not receive a contribution from the identity matrix in \eqref{M_decomposed_eq}, so
  \begin{align*}
    R_{\{a,b\},\{a,c\}} & = a_2 + a_3 \\
    & = C_3 - C_2^2 .
  \end{align*}
Finally, a coordinate pair of the form $(\{a,b\}, \{c,d\})$ (i.e. no shared components) receives a contribution only from the $\vec x_{d-1} \vec x_{d-1}^T$ term in \eqref{M_decomposed_eq}:
  \begin{align*}
    R_{\{a,b\}, \{c,d\}} & = a_2 \\
    & = C_4 - C_2^2. \qedhere
  \end{align*}
\end{proof}

The useful characterization of $M$ in Proposition \ref{decomposing_M} will allow us to compute the determinant of $M$ when combined with the following lemmas:

\begin{lemma}[Matrix Determinant Lemma] \label{matrixdetlemma}
  Let $W$ be an invertible $n \times n$ matrix and let $U, V$ be $n \times m$ matrices.  Then
    \[ \det(W + UV^T) = \det(W) \det(I_m + V^T W^{-1} U) \, .\]
\end{lemma}

\begin{proof}
  See \cite{harville}*{Theorem 18.1}.
\end{proof}

\begin{lemma}[Generalized Sherman-Morrison-Woodbury Identity] \label{gsmwi}
  Let $W$ be an invertible $n \times n$ matrix and for $i = 1, \dots, L$ let $U_i, V_i$ be $n \times m$ matrices.  Define the $Lm \times Lm$ matrix $X$ by
  \[X = \left[ \begin{array}{cccc}
      I_m + V_1^T W^{-1} U_1 & V_1^T W^{-1} U_2 & \dots & V_1^T W^{-1} U_L \\
      V_2^T W^{-1} U_1 & I_m + V_2^T W^{-1} U_2 & \dots & V_2^T W^{-1} U_L \\
      \vdots & \vdots & \ddots & \vdots \\
      V_L^T W^{-1} U_1 & V_L^T W^{-1} U_2 & \dots & I_m + V_L^T W^{-1} U_L
  \end{array} \right] . \]
If $X$ is invertible, then the matrix $\left( W + \sum_{i=1}^L U_i V_i^T \right)$ is invertible, and its inverse is given by
   \begin{equation*}
   \left( W + \sum_{i=1}^{L} U_i V_i^T \right)^{-1} = W^{-1} - W^{-1}[\begin{array}{cccc} U_1 & \dots & U_L \end{array}] X^{-1} [\begin{array}{cccc} V_1^T & \dots & V_L^T \end{array}]^T W^{-1} \, .
   \end{equation*}
\end{lemma}

\begin{proof}
  See \cite{batista}.
\end{proof}

In particular, with $L = 1$ in Lemma \ref{gsmwi}, we obtain the following:
\begin{lemma}[Woodbury Matrix Identity] \label{woodbury}
  Let $W$ be an invertible $n \times n$ matrix and let $U, V$ be $n \times m$ matrices.  Define $X = I_m + V^T W^{-1} U$.  If $X$ is invertible, then $W + UV^T$ is invertible, and
  \[(W + UV^T)^{-1} = W^{-1} - W^{-1}UX^{-1} V^T W^{-1} \, .\]
\end{lemma}
The basic strategy for computing $\det(M)$ will be to use the Matrix Determinant Lemma several times to trade the products $Q Q^T$ and $\vec x_{d-1} \vec x_{d-1}^T$ for their lower-rank counterparts, $Q^T Q$ and $\vec x_{d-1}^T \vec x_{d-1}$.  Executing this plan will require use of the generalized Sherman-Morrison-Woodbury and Woodbury Matrix Identities.

Finally, before computing $\det(M)$, we remark that if $k = 2$, we have $C_3 = C_4 = 0$ and therefore $a_3 = 0$ in Lemma \ref{decomposing_M}.  For technical reasons, this will require us to approach the computation differently when $k = 2$.  However, the formula given in Lemma \ref{determinant_calc} will still hold in this case, even though the proof is slightly different.

\begin{proof}[Proof of Lemma \ref{determinant_calc}]

We first assume that $k \geq 3$.  Recalling the definitions of $a_1, a_2, $ and $a_3$ in Proposition \ref{decomposing_M}, we have
  \begin{equation*}
    a_3 = C_3 - C_4 = \frac{k(k-1)(k-2)}{v(v-1)(v-2)} \left( 1 - \frac{k-3}{v-3} \right)
  \end{equation*}
so $a_3 > 0$.  Similarly, 
  \begin{align*} 
    a_1 & = C_2 - 2 C_3 + C_4 = \frac{k(k-1)}{v(v-1)} \left( 1 - 2 \frac{k-2}{v-2} + \frac{(k-2)(k-3)}{(v-2)(v-3)} \right) 
  \end{align*}
and since 
  \begin{align*}
    0 & < [(v-3) - (k-2)]^2 + (v-k - 1) \\
    & = (v-3)(v-2) -2(k-2)(v-3)+(k-2)(k-3)
  \end{align*}
it follows that $a_1 > 0$ as well.  We define a constant $w$ that will appear in several places:
  \begin{equation} \label{def_w}
    w = \frac{a_3}{a_1}(v-2) + 1
  \end{equation}
Since $a_3 > 0$ and $a_1 > 0$, it follows that $w \geq 1$.

Starting with the decomposition in Proposition \ref{decomposing_M}, we set
  \begin{equation} \label{def_E}
    E = a_1 I_{d-1} + a_3 Q Q^T
  \end{equation}
so that we have
  \[ M = E + a_2 \vec x_{d-1} \vec x_{d-1}^T \, .\]
Once we have shown that $E$ is invertible, by the Matrix Determinant Lemma we will have
  \begin{equation} \label{det_M_big_goal}
    \det(M) = \det(E) (1 + a_2 \vec x_{d-1}^T E^{-1} \vec x_{d-1}).
  \end{equation}
 This breaks the computation of $\det(M)$ into two smaller computations; we will handle the computation of $1 + \vec x_{d-1}^T E^{-1} \vec x_{d-1}$ first.  Since $E = a_1 I_{d-1} + a_3 Q Q^T$, so long as the matrix 
  \begin{equation} \label{defn_matrix_G}
    G = I_v + \frac{a_3}{a_1} Q^T Q
  \end{equation}
 is invertible, applying the Woodbury Matrix Identity to \eqref{def_E} yields
  \begin{equation} \label{det_M_E_inverse}
    E^{-1} = a_1^{-1}I_{d-1} - a_1^{-2} a_3 Q G^{-1} Q^T.
  \end{equation}
Applying \eqref{QT_Q} to the $Q^T Q$ expression in \eqref{defn_matrix_G} and using $w$ as in \eqref{def_w} gives
  \begin{equation} \label{det_M_decomposing_G}
    G = w I_v + \frac{a_3}{a_1} \vec x_v \vec x_v^T - \frac{a_3}{a_1} \vec y_v \vec y_v^T.
  \end{equation}

To argue that $G$ is invertible (hence, that $E$ is), and to compute $G^{-1}$, we use the generalized Sherman-Morrison-Woodbury Identity on \eqref{det_M_decomposing_G}.  Here, the matrix $X$ in Lemma \ref{gsmwi} is the $2 \times 2$ matrix which can be computed using \eqref{xrt_xr} and \eqref{yrt_yr}:
  \begin{equation} \label{det_M_def_X} X = \left[ \begin{array}{cc}
      1 + \frac{1}{w} \frac{a_3}{a_1} \vec x_v^T \vec x_v & - \frac{1}{w} \frac{a_3}{a_1} \vec x_v^T \vec y_v \\
      \frac{1}{w} \frac{a_3}{a_1} \vec x_v^T \vec y_v & 1 - \frac{1}{w} \frac{a_3}{a_1} \vec y_v^T \vec y_v
  \end{array} \right] = \left[ \begin{array}{cc} 
      1 + v\frac{a_3}{a_1 w} & - 2\frac{a_3}{a_1 w} \\ 
      2 \frac{a_3}{a_1 w} & 1 - 2 \frac{a_3}{a_1 w}
   \end{array} \right] \end{equation}
We note that 
  \begin{align*}
    \det(X) & = \left(1 + v \frac{a_3}{a_1 w} \right) \left(1 - 2 \frac{a_3}{a_1 w} \right) + 4\left( \frac{a_3}{a_1 w} \right)^2  \\
    & = \left( \frac{a_3}{a_1 w} \right)^2 \left( \left(\frac{a_1 w}{a_3} + v \right) \left( \frac{a_1 w}{a_3} -2 \right) + 4 \right).
  \end{align*}
Since $\frac{a_1 w}{a_3} = v-2 + \frac{a_1}{a_3} \geq 2$, it follows that this determinant is nonzero.  Hence, $X$ is invertible, which implies that $G$ is invertible, and therefore $E$ is invertible, justifying the use of \eqref{det_M_big_goal}. 

By inverting the $2 \times 2$ matrix $X$ and applying the generalized Sherman Morrison-Woodbury Identity, after some algebra we have
  \begin{equation} 
     G^{-1} = \frac{1}{w} \left( I_v -\frac{1}{(\frac{a_1 w}{a_3} + v)(\frac{a_1 w}{a_3} - 2)+4}  \left[ \begin{array}{cc} \vec x_v & - \vec y_v \end{array} \right]
      \left[ \begin{array}{cc} \frac{a_1 w}{a_3} - 2 & 2 \\ -2 & \frac{a_1 w}{a_3} + v \end{array} \right]
      \left[ \begin{array}{c} \vec x_v^T \\ \vec y_v^T \end{array} \right]
      \right) \label{det_M_G_inverse}
  \end{equation}
giving us an explicit formula for $G^{-1}$.  By \eqref{det_M_E_inverse}, this also gives an explicit formula for $E^{-1}$.  The right half of the computation in \eqref{det_M_big_goal} can be rewritten using \eqref{det_M_E_inverse} to obtain
  \begin{align*}
     & 1 + a_2 \vec x_{d-1}^T E^{-1} \vec x_{d-1} = 1 + \frac{a_2}{a_1} \vec x_{d-1}^T \vec x_{d-1} - \frac{a_2 a_3}{a_1^2} \vec x_{d-1}^T Q G^{-1} Q^T \vec x_{d-1}
  \end{align*}
  
and by using \eqref{xT_Q} to replace $\vec x^T_{d-1} Q$ and $Q^T \vec x_{d-1}$ we have
  \begin{equation} 
    1 + a_2 \vec x_{d-1}^T E^{-1} \vec x_{d-1} = 1 + \frac{a_2}{a_1} \vec x_{d-1}^T \vec x_{d-1} - \frac{a_2 a_3}{a_1^2} [(v-1)\vec x_v^T - \vec y_v^T] G^{-1} [(v-1) \vec x_v - \vec y_v] \label{det_M_right_side_1}
  \end{equation}
which can be computed due to the explicit formula for $G^{-1}$ given in \eqref{det_M_G_inverse}.  For convenience of notation, we set
  \begin{align*}
    U & = \left[ \begin{array}{cc} \vec x_v & - \vec y_v \end{array} \right] \\
    H & = \left[ \begin{array}{cc} \frac{a_1 w}{a_3} - 2 & 2 \\ -2 & \frac{a_1 w}{a_3} + v \end{array} \right] \\
    V^T & = \left[ \begin{array}{c} \vec x_v^T \\ \vec y_v^T \end{array} \right]
  \end{align*}
since these matrices appear in the more complicated portion of $G^{-1}$.  To expand the product in \eqref{det_M_right_side_1}, we observe four useful calculations that make use of \eqref{xrt_xr} and \eqref{yrt_yr}:
  \begin{align*}
    \vec x_v^T U H V^T \vec x_v & = (v-2)\left( \frac{a_1 w}{a_3}(v+2) - 2v \right) \\
    \vec x_v^T U H V^T \vec y_v & = 2(v-2) \left( \frac{a_1 w}{a_3} - 2 \right) \\
    \vec y_v^T U H V^T \vec x_v & = 2(v-2) \left( \frac{a_1 w}{a_3} - 2 \right) \\
    \vec y_v^T U H V^T \vec y_v & = 8-4v 
  \end{align*}
Using these calculations in \eqref{det_M_right_side_1}, along with \eqref{xrt_xr} and \eqref{yrt_yr} again and a great deal of algebra, we have
  \begin{align}
    \nonumber & 1 + a_2\vec x_{d-1}^T E^{-1} \vec x_{d-1} = 1 + \frac{a_2}{a_1} \bigg[d-1 - \frac{w-1}{w} \bigg\{v^2 - 3 - \frac{v-2}{(\frac{a_1 w}{a_3} + v ) (\frac{a_1 w}{a_3} -2 ) + 4}  \\
    & \qquad \quad \times \left( (v-3)(v^2 + v - 4) + \frac{a_1}{a_3} (v-1)(v+3) \right) \bigg\} \bigg]. \label{det_M_right_final}
  \end{align}

This yields a formula for the second factor on the right-hand side of \eqref{det_M_big_goal}.

To find a formula the first factor on the right-hand side of \eqref{det_M_big_goal}, we seek to compute $\det(E)$.  By the Matrix Determinant Lemma, we have 
  \begin{align*} 
    \det(E) & = a_1^{d-1} \det \left(I_{d-1} + \frac{a_3}{a_1} Q Q^T \right)  = a_1^{d-1} \det \left( I_v + \frac{a_3}{a_1} Q^T Q \right)
  \end{align*}
so from \eqref{QT_Q}, we have
  \begin{align*}
    \det(E) &= a_1^{d-1} \det \left( w I_v + \frac{a_3}{a_1} \vec x_v \vec x_v^T - \frac{a_3}{a_1} \vec y_v \vec y_v^T \right) .
  \end{align*}
We set 
  \begin{equation} \label{det_M_def_F}
    F = w I_v + \frac{a_3}{a_1} \vec x_v \vec x_v^T
  \end{equation}
and we note that if $F$ is invertible, then by the Matrix Determinant Lemma, we have
  \begin{align} \label{det_M_det_E_2}
    \det(E) & = a_1^{d-1} \det(F) \left( 1 - \frac{a_3}{a_1} \vec y_v^T F^{-1} \vec y_v \right).
  \end{align}
To establish that $F$ is invertible and to compute $F^{-1}$, we use the Woodbury Matrix Identity on \eqref{det_M_def_F}.  Because 
  \[1 + \frac{a_3}{a_1 w} \vec x_v^T \vec x_v = 1 + \frac{a_3 v}{a_1 w} > 0 \]
it follows from Lemma \ref{woodbury} that $F$ is invertible, so that the use of \eqref{det_M_det_E_2} is indeed justified.  Moreover, from this lemma we obtain
  \begin{equation*} 
    F^{-1} = \frac{1}{w} \left( I_v - \frac{1}{\frac{a_1 w}{a_3}+v} \vec x_v \vec x_v^T \right).
  \end{equation*}
We combine this with \eqref{yrt_yr} and find, after some algebra, that
  \begin{equation} \label{det_M_det_E_3}
    1 - \frac{a_3}{a_1} \vec y_v^T F^{-1} \vec y_v = 1 - \frac{a_3}{w a_1} \left( 2 - \frac{4}{\frac{a_1 w}{a_3} + v} \right) \, . 
  \end{equation}

To find $\det(F)$, we apply the Matrix Determinant Lemma to \eqref{det_M_def_F} to see that
  \begin{align} 
    \det(F) & = w^v \det \left(I_v + \frac{a_3}{a_1 w} \vec x_v \vec x_v^T  \right) \nonumber \\
    & = w^{v-1} \left( w + \frac{a_3}{a_1}v \right). \label{det_M_det_F}
  \end{align}
By substituting the results of \eqref{det_M_det_F} and \eqref{det_M_det_E_3} into \eqref{det_M_det_E_2} and simplifying, we find that 
  \begin{equation} \label{det_M_det_E}
    \det(E) = a_1^{d-1} w^{v-2} \left( 2w^2 - w - 2 \frac{a_3}{a_1}(w-1) \right) .
  \end{equation}
From here, \eqref{det_M_det_E} and \eqref{det_M_right_final} yield the two factors of $\det(M)$ in \eqref{det_M_big_goal}.  We multiply these together and substitute the definition of $w$ in \eqref{def_w}.  Then, we substitute the values of $a_1, a_2, a_3$ in Proposition \ref{decomposing_M}; following this, using the definition of the $C_i$ constants and simplifying yields \eqref{det_M}.

  Finally, in the case where $k = 2$, we note that \eqref{det_M} reduces to the particularly simple expression
    \begin{equation} \label{det_M_k_2}
      \det(M) = \frac{1}{d^d}.
    \end{equation}
When $k = 2$, the coefficients $a_1, a_2, a_3$ are
  \begin{align*}
    a_1 & = C_2 = d^{-1}, \\
    a_2 & = -C_2^2 = -d^{-2}, \\
    a_3 & = 0.
  \end{align*}
The preceding proof does not work since $a_3$ appears in many denominators.  To verify that the formula in \eqref{det_M_k_2} still holds, we reconsider the decomposition of $M$ in Proposition \ref{decomposing_M}.  In this case,
  \begin{equation} \label{M_decomp_k_2}
    M = a_1 I_{d-1} + a_2 \vec x_{d-1} \vec x_{d-1}^T
  \end{equation}
so the determinant is much more straightforward than the case where $k \geq 3$. 
In particular, since $a_1 \neq 0$ we can apply the Matrix Determinant Lemma to \eqref{M_decomp_k_2}.  This gives
  \begin{align*}
    \det(M) & = \det(a_1 I_{d-1}) \left(1 + \frac{a_2}{a_1} \vec x_{d-1}^T \vec x_{d-1} \right) \\
    & = (d^{-1})^{d-1} \left( 1 - \frac{d^{-2}}{d^{-1}}(d-1) \right) \\
    & = d^{-d}
  \end{align*}
which matches \eqref{det_M_k_2} and completes the proof.
\end{proof}

\section{Proof of Main Theorem} \label{S:proof}

Our next task is to find suitable lower and upper bounds for the integral used to compute $ \PB_{v,k}^{(t)}(\vec 0, \vec 0)$.  With $D_1$ and $D_2$ defined as in Proposition \ref{scaling_sets}, we define two quantities of interest:
  \begin{align*}
    L(v, k, t, \delta) & = [1 + t^2(d \delta)^6]^{-1/2} \left[ 1 - \frac{1}{3}(d \delta)^4 \right]^t [1 - e^{-\frac{1}{2} t (D_1\delta)^2}]^{(d-1)/2}\\
    U(v, k, t, \delta) & = \left[ 1 + \frac{1}{4}(d \delta)^6 \right]^{t/2} \left[ 1 + \frac{1}{3} ( d \delta)^4 \right]^t  [1 - e^{- t (2 D_2 \delta)^2}]^{(d-1)/2}
  \end{align*}
  
\begin{thm} \label{return_prob_estimates_1}
 Suppose that $\delta < k^{-2} \nk^{-2} \left[ \frac{1}{6 \cdot 96^2} \left( \frac{2 \pi}{k-1} \right)^4 \right]$, and let $t \geq 2$ be any integer such that $t < 2 (d \delta)^{-3}$. If $t \frac{k(k-1)}{v(v-1)}$ is not an integer, then
  \begin{equation} \label{return_prob_bound_awful}
    \PB_{v,k}^{(t)}(\vec 0, \vec 0) = 0.
  \end{equation}
If $t \frac{k(k-1)}{v(v-1)}$ is an integer but $t \frac{k}{v}$ is not, then
  \begin{equation} \label{return_prob_bound_bad}
    \PB_{v,k}^{(t)}(\vec 0, \vec 0) \leq \exp \left( - \nk^{-1} \frac{11}{768} t \delta^2 \right) \, .
  \end{equation}
Finally, if both $t \frac{k(k-1)}{v(v-1)}$ and $t\frac{k}{v}$ are integers, then  
  \begin{equation} \label{return_prob_upper}
    \PB_{v,k}^{(t)}(\vec 0, \vec 0) \leq \frac{(k-1)^{v-1}}{\sqrt{(2 \pi t)^{d-1} \det(M)}} U(v , k, t, \delta) + e^{- \nk^{-1} \frac{11}{768} t \delta^2}
  \end{equation}
and
  \begin{equation} \label{return_prob_lower}
    \PB_{v,k}^{(t)}(\vec 0, \vec 0) \geq\frac{(k-1)^{v-1}}{\sqrt{(2 \pi t)^{d-1} \det(M)}} L(v , k, t, \delta) - e^{- \nk^{-1} \frac{11}{768} t \delta^2} \, .
  \end{equation}

\end{thm}

\begin{rmk}
  In the sequel, $\delta$ will be chosen to vary with $t$ in such a way that $t \delta^2$ diverges to infinity.  This will cause the bound in \eqref{return_prob_bound_bad} to tend to $0$, which reflects the fact that a balanced incomplete block design can only exist when $t \frac{k}{v}$ is an integer as shown in \eqref{BIBD_relation1}.  The terms $U(v, k, t, \delta)$ and $L(v, k, t, \delta)$ will also approach $1$, which will cause \eqref{return_prob_upper} and \eqref{return_prob_lower} to yield the asymptotics for the return probability of the random walk $Y_t$.  This will then give the asymptotics for the number of balanced incomplete block designs as $t$ increases.
\end{rmk}

\begin{rmk}
Since $k \geq 2$ and $v-k \geq 2$, we have $\nk \geq \binom{v}{2} = d$; hence, our assumption on $\delta$ implies in particular that $\delta < d^{-1}$, which will be referenced throughout the proof. 
\end{rmk}

We require some technical lemmas before proving Theorem \ref{return_prob_estimates_1}. 

\begin{defn} \label{defn_alpha_wdl}
  For any positive number $t$ and any complex number $z$, we set $\beta(z) = \imag(z)/\real(z)$ and $\alpha(z, t) = 1 - \binom{t}{2} \beta(z)^2$.
\end{defn}

\begin{lemma} \label{z^t_to_z}
  Let $t \geq 2$ be an integer, and let $z \in \C$ with $\real(z) > 0$ and $\alpha(z, t) > 0$.
  Then
      \begin{equation} \label{appendix_2}
        \real(z^t) \leq \real(z)^t \left( 1 + \left( \frac{\imag(z)}{\real(z)} \right)^2 \right)^{t/2}
      \end{equation}      
    and 
      \begin{equation} \label{appendix_5}
        \real(z^t) \geq \real(z)^t \left(1 + \left[ \frac{\imag(z)}{\real(z)} \right]^2 \right)^{t/2} \left(1 + \left[ \frac{t}{\alpha(z,t)} \right]^2 \left[ \frac{\imag(z)}{\real(z)} \right]^2 \right)^{-1/2}.
      \end{equation}
\end{lemma}

\begin{proof}
  This requires only trivial modifications to parts (i) and (iv) of Proposition A.2 in \cite{wdl_levin}.
\end{proof}

\begin{lemma} \label{gaussian_estimate}
  Let $\rho$ be a positive real number.  Then
    \[\sqrt{2 \pi(1 - e^{- \rho^2/2})} \leq \int_{- \rho}^{\rho} e^{-\frac{1}{2} x^2} \ud x \leq \sqrt{2 \pi(1 - e^{-\rho^2})}.\]
\end{lemma}

\begin{proof}
  Using the standard trick of multiplying two copies of the integral together, using Fubini's Theorem, and converting to polar coordinates, we have
  \[\int_0^{\rho} 2 \pi r e^{-\frac{1}{2} r^2} \ud r < \int_{[-\rho, \rho]^2} e^{-\frac{1}{2}(x^2 + y^2)} \ud y \, \ud x < \int_0^{\sqrt{2} \rho} 2 \pi r e^{-\frac{1}{2} r^2} \ud r \]
so computing the left and right integrals and taking square roots gives the result.
\end{proof}

\begin{proof}[Proof of Theorem \ref{return_prob_estimates_1}]
   We first consider the case where $t \frac{k(k-1)}{v(v-1)}$ is not an integer.  From the definitions of $X_t$ and $Y_t$, since $X_t$ is supported on $\Z^d$ then it is trivially only possible to have $Y_t = \vec 0$ if $t \frac{k(k-1)}{v(v-1)} \in \Z$, which establishes \eqref{return_prob_bound_awful}.

  When $t \frac{k(k-1)}{v(v-1)} \in \Z$, we recall from \eqref{fourier_inversion} that 
  \begin{equation*}
    \PB_{v,k}^{(t)}(\vec 0, \vec 0) = (2 \pi)^{-d} \int_{[-\pi, \pi]^d} \Phi_Y (\vec \theta)^t \, d \vec \theta \, .
  \end{equation*}
If $t \frac{k}{v} \not \in \Z$, then from \eqref{integral_breakup_4}, we have
  \begin{equation*}
    \PB_{v,k}^{(t)}(\vec 0, \vec 0) = (2 \pi)^{-d} \int_{R_{B, \equiv}^{\delta} \cup R_{C, \equiv}^{\delta}} \Phi_Y (\vec \theta)^t \, d \vec \theta
  \end{equation*}
whence Proposition \ref{bad_region_decay} gives rise to \eqref{return_prob_bound_bad}.  

If instead $t \frac{k}{v} \in \Z$, \eqref{integral_breakup_5} implies that
  \begin{align*}
   & \left| \PB_{v,k}^{(t)}(\vec 0, \vec 0) -  (2 \pi)^{-d}(k-1)^{v-1} \int_{T_{\vec 0}^{\delta}} \Phi_Y(\vec \theta)^t \ud \vec \theta \right| = \left| (2 \pi)^{-d} \int_{R_{B, \equiv}^{\delta} \cup R_{C, \equiv}^{\delta}} \Phi_Y(\vec \theta)^t \ud \vec \theta \right|
  \end{align*}
so that Proposition \ref{bad_region_decay} yields 
  \begin{equation*}
     \left| \PB_{v,k}^{(t)}(\vec 0, \vec 0) -  (2 \pi)^{-d}(k-1)^{v-1} \int_{T_{\vec 0}^{\delta}} \Phi_Y(\vec \theta)^t \ud \vec \theta \right| \leq e^{- \nk^{-1} \frac{11}{768} t \delta^2} \, .
  \end{equation*}
Therefore, to prove \eqref{return_prob_upper} and \eqref{return_prob_lower}, it will suffice to show that
  \begin{equation} \label{final_theorem_goal_1}
    \frac{L(v,k,t,\delta)}{\sqrt{(2 \pi t)^{d-1} \det(M)}} \leq (2 \pi)^{-d} \int_{T_{\vec 0}^{\delta}} \Phi_Y (\vec \theta)^t \ud \vec \theta  \leq \frac{U(v,k,t,\delta)}{\sqrt{(2 \pi t)^{d-1} \det(M)}} \, .
  \end{equation}
Moreover, since $\Phi_Y(- \vec \theta)$ and $\Phi_Y(\vec \theta)$ are complex conjugates and $T_{\vec 0}^{\delta}$ is closed under negation, we have
  \begin{equation} \label{final_theorem_goal_2}
    \int_{T_{\vec 0}^{\delta}} \Phi_Y(\vec \theta)^t \ud \vec \theta = \int_{T_{\vec 0}^{\delta}} \real(\Phi_Y( \vec \theta)^t) \ud \vec \theta.
  \end{equation}
Our strategy will be to relate $\real(\Phi_Y(\vec \theta)^t)$ to $[\real(\Phi_Y(\vec \theta))]^t$ by using Lemma \ref{z^t_to_z}.

Let $t \geq 2$ be an integer and let $\beta(z), \alpha(z, t)$ be as in Definition \ref{defn_alpha_wdl}.  From Lemma \ref{good_region_bounds}, for $\vec \theta \in T_{\vec 0}^{\delta}$ we have
  \begin{equation} \label{beta_upper_bound}
    |\beta(\Phi_Y(\vec \theta))| \leq \frac{(d \delta)^3 / 6}{1/3} = \frac{(d \delta)^3}{2}
  \end{equation}
Since by hypothesis $t < 2(d \delta)^{-3}$, it follows that $\binom{t}{2} \beta(\Phi_Y(\vec \theta))^2 \leq \binom{t}{2} \frac{(d \delta)^6}{4} < \frac{1}{2}$, whence $\alpha(\Phi_Y(\vec \theta), t) > \frac{1}{2}$.  In particular, since $\real(\Phi_Y(\vec \theta)) > 0$ by \eqref{phi_real_lower_bound} and since $\alpha(\Phi_Y(\vec \theta), t) > 0$, we can make full use of Lemma \ref{z^t_to_z}.  From \eqref{appendix_2} and \eqref{beta_upper_bound} we have 
  \begin{equation} \label{transition_upper_bound}
    \real(\Phi_Y^t(\vec \theta)) \leq \left[ \real(\Phi_Y(\vec \theta)) \right]^t \left( 1 + \frac{(d \delta)^6}{4} \right)^{t/2}
  \end{equation}
and if $\beta$ and $\alpha$ denote $\beta(\Phi_Y(\vec \theta))$ and $\alpha(\Phi_Y(\vec \theta), t)$ respectively, then from \eqref{appendix_5} we have
  \begin{align}
    \real(\Phi_Y(\vec \theta)^t) & \geq \left[\real(\Phi_Y(\vec \theta)) \right]^t \left( 1 + \beta ^2 \right)^{t/2} \left(1 + t^2 \left[\frac{\beta}{\alpha} \right]^2 \right)^{-\frac{1}{2}}  \nonumber \\
   & \geq \left[\real(\Phi_Y(\vec \theta)) \right]^t \left(1 + t^2 \left[\frac{\beta}{\alpha} \right]^2 \right)^{-\frac{1}{2}}. \label{transition_lower_bound}
  \end{align}
Since $\alpha(\Phi_Y(\vec \theta), t) \geq 1/2$ and $\beta(\Phi_Y(\vec \theta)) \leq (d \delta)^3 / 2$, it follows that 
  \[ \left[ \frac{\beta(\Phi_Y(\vec \theta))}{\alpha(\Phi_Y(\vec \theta), t)} \right]^2 \leq (d \delta)^6  \]
so \eqref{transition_upper_bound} and \eqref{transition_lower_bound} combine to give
  \begin{equation}
    [1 + t^2 (d \delta)^6 ]^{-1/2} \int_{T_{\vec 0}^{\delta}} \left[ \real(\Phi_Y(\vec \theta)) \right]^t \ud \vec \theta
    \leq \int_{T_{\vec 0}^{\delta}} \real(\Phi_Y(\vec \theta)^t) \ud \vec \theta
    \leq \left[1 + \frac{(d \delta)^6}{4} \right]^{t/2} \int_{T_{\vec 0}^{\delta}} \left[ \real(\Phi_Y(\vec \theta)) \right]^t \ud \vec \theta.  \label{power_traded}
  \end{equation} 
  
 From Lemma \ref{good_region_bounds}, we see that there exists a function $\varepsilon_1: T_{\vec 0}^{\delta} \to \R$ such that for $\vec \theta \in T_{\vec 0}^{\delta}$,
  \[ \left[\real(\Phi_Y(\vec \theta)) \right]^t = e^{-\frac{1}{2} \vec \theta^T N \vec \theta} ( 1 + \varepsilon_1(\vec \theta))^t \]
and $|\varepsilon_1(\vec \theta)| < \frac{1}{6} (d \delta)^4 e^{\frac{1}{2} (d \delta)^2}$.  Since our assumptions imply that $d \delta < 1$, it follows that $e^{\frac{1}{2} (d \delta)^2} < 2$, so $|\varepsilon_1(\vec \theta)| < \frac{1}{3}(d \delta)^4$.  Hence, we have
  \begin{equation*}
   e^{-\frac{t}{2} \vec \theta^T N \vec \theta} \left[ 1 - \frac{1}{3}(d \delta)^4 \right]^t \leq \left[ \real(\Phi_Y(\vec \theta)) \right]^t  \leq e^{-\frac{t}{2} \vec \theta^T N \vec \theta} \left[ 1 + \frac{1}{3}(d \delta)^4 \right]^t 
  \end{equation*}
and substituting these bounds into \eqref{power_traded} gives
  \begin{align}
   & [1 + t^2 (d \delta)^6 ]^{-1/2}\left[1 - \frac{1}{3}(d \delta)^4 \right]^t \int_{T_{\vec 0}^{\delta}}  e^{-\frac{t}{2} \vec \theta^T N \vec \theta} \ud \vec \theta \nonumber \\
   & \qquad \quad \leq \int_{T_{\vec 0}^{\delta}} \real(\Phi_Y(\vec \theta)^t) \ud \vec \theta \nonumber \\
   & \qquad \quad \leq \left[1 + \frac{(d \delta)^6}{4} \right]^{t/2} \left[ 1 + \frac{1}{3} (d \delta)^4 \right]^t \int_{T_{\vec 0}^{\delta}} e^{-\frac{t}{2} \vec \theta^T N \vec \theta} \ud \vec \theta.  \label{gaussian_integral}
  \end{align} 
To verify \eqref{final_theorem_goal_1} (and thus complete the proof), by \eqref{gaussian_integral} and \eqref{final_theorem_goal_2} it suffices to show that 
  \begin{equation}
     \frac{[1 - e^{-\frac{1}{2} t (D_1\delta)^2}]^{(d-1)/2}}{\sqrt{\det(M)}} (2 \pi) \left( \frac{2 \pi}{t} \right)^{\frac{d-1}{2}}
     \leq \int_{T_{\vec 0}^{\delta}} e^{-\frac{t}{2} \vec \theta^T N \vec \theta} \ud \vec \theta
     \leq  \frac{[1 - e^{ -t (2D_2 \delta)^2}]^{(d-1)/2}}{\sqrt{\det(M)}} (2 \pi) \left( \frac{2 \pi}{t} \right)^{\frac{d-1}{2}} \label{last_reduction_really}
  \end{equation}
so we now turn our attention to the integral in the middle.

Using Lemma \ref{obnoxious_reparametrization}, we see that
  \[ 2 \pi \int_{[-\delta, \delta]^{d-1}} e^{-\frac{t}{2} \vec \mu^T M \vec \mu} \ud \vec \mu \leq \int_{T_{\vec 0}^{\delta}} e^{-\frac{t}{2} \vec \theta^T N \vec \theta} \ud \vec \theta \leq 2 \pi \int_{[-2\delta, 2\delta]^{d-1}} e^{-\frac{t}{2} \vec \mu^T M \vec \mu} \ud \vec \mu. \]
We recall from Corollary \ref{M_positive_definite} that $M$ is positive definite, so there is a symmetric, positive definite matrix $P$ for which $P^2 = M$. Hence,
  \[2 \pi \int_{[-\delta, \delta]^{d-1}} e^{-\frac{t}{2} \vec \mu^T M \vec \mu} \ud \vec \mu = 2 \pi \int_{[-\delta, \delta]^{d-1}} e^{-\frac{t}{2} (P\vec \mu)^T (P \vec \mu)} \ud \vec \mu  \]
so if we apply a change of variables with $\vec \eta = P \vec \mu$, we have
  \[ 2 \pi \int_{[-\delta, \delta]^{d-1}} e^{-\frac{t}{2} \vec \mu^T M \vec \mu} \ud \vec \mu =  \frac{2 \pi}{\det(P)} \int_{P[-\delta, \delta]^{d-1}} e^{-\frac{t}{2} \vec \eta^T \vec \eta} \ud \vec \eta \]
and similarly,
  \[ 2 \pi \int_{[-2\delta, 2\delta]^{d-1}} e^{-\frac{t}{2} \vec \mu^T M \vec \mu} \ud \vec \mu =  \frac{2 \pi}{\det(P)} \int_{P[-2\delta, 2\delta]^{d-1}} e^{-\frac{t}{2} \vec \eta^T \vec \eta} \ud \vec \eta \, .\]
Since the integrand is positive, using Proposition \ref{scaling_sets} gives
  \begin{equation*}
  \frac{2 \pi}{\det(P)} \int_{[-D_1 \delta, D_1 \delta]^{d-1}} e^{-\frac{t}{2} \vec \eta^T \vec \eta} \ud \vec \eta
  < \int_{[-\delta, \delta]^{d-1}} e^{-\frac{t}{2} \vec \mu^T M \vec \mu} \ud \vec \mu 
  < \frac{2 \pi}{\det(P)} \int_{[-2D_2 \delta, 2D_2 \delta]^{d-1}} e^{-\frac{t}{2} \vec \eta^T \vec \eta} \ud \vec \eta.
  \end{equation*}
Making one last change of variables with $\vec \nu = \sqrt{t} \vec \eta$ on the upper and lower bounds yields
  \begin{align*}
  &  \frac{2 \pi}{\det(P)} (\sqrt{t})^{-(d-1)} \int_{[-D_1\sqrt{t} \delta, D_1\sqrt{t} \delta]^{d-1}} e^{-\frac{1}{2} \vec \nu^T \vec \nu} \ud \vec \nu \\
  & \qquad \quad < \int_{[-\delta, \delta]^{d-1}} e^{-\frac{t}{2} \vec \mu^T M \vec \mu} \ud \vec \mu   \\
  & \qquad \quad < \frac{2 \pi}{\det(P)} (\sqrt{t})^{-(d-1)} \int_{[-2D_2 \sqrt{t} \delta, 2D_2 \sqrt{t} \delta]^{d-1}} e^{-\frac{1}{2} \vec \nu^T \vec \nu} \ud \vec \eta.
  \end{align*}
Since $\vec \nu^T \vec \nu = \sum \nu_{\{i,j\}}^2$, we can regard the integrals in the lower and upper bounds as the product of $d-1$ integrals of the form $\int e^{-\frac{1}{2} x^2} \ud x$.  Using the estimates in Lemma \ref{gaussian_estimate} gives
  \begin{align*}
  &  \frac{2 \pi}{\det(P)} (\sqrt{t})^{-(d-1)} \left( \sqrt{2 \pi(1 - e^{- \frac{1}{2} t (D_1 \delta)^2})} \right)^{d-1} \\
  & \qquad \quad < \int_{[-\delta, \delta]^{d-1}} e^{-\frac{t}{2} \vec \mu^T M \vec \mu} \ud \vec \mu   \\
  & \qquad \quad < \frac{2 \pi}{\det(P)} (\sqrt{t})^{-(d-1)} \left( \sqrt{2 \pi(1 - e^{- t (2D_2 \delta)^2})} \right)^{d-1} 
  \end{align*}
and since $\det(P) = \sqrt{\det(M)}$, this yields \eqref{last_reduction_really} and completes the proof.
\end{proof}

\begin{proof}[Proof of Theorem \ref{number_BIBD_matrices}]
The main point of the proof is to allow $t$ and $\delta$ to vary in such a way that in \eqref{return_prob_upper} and \eqref{return_prob_lower}, the $U$ and $L$ terms tend to $1$, while the error terms in \eqref{return_prob_bound_bad}, \eqref{return_prob_upper}, and \eqref{return_prob_lower} tend to $0$.  For a fixed $v$ and $k$, we claim that setting $\delta = t^{-5/12}$ will accomplish this.
 
We first note that for sufficiently large $t$, $\delta$ is arbitrarily small and thus $\delta < k^{-2} \nk^{-2} \left[ \frac{1}{6 \cdot 96^2} \left( \frac{2 \pi}{k-1} \right)^4 \right]$ eventually holds.  Similarly, since $(d \delta)^{-3} = d^{-3} t^{5/4}$, for sufficiently large $t$ we have $t < 2(d \delta)^{-3}$.  This allows all parts of Theorem \ref{return_prob_estimates_1} to be used. 

We turn our attention to the terms in square brackets in $L$ and $U$.  Since $t^2 \delta^6 = t^{-1/2}$, it follows that $[1 + t^2 (d \delta)^6]^{-1/2} \to 1$ as $t \to \I$.  For any constant $C$ that does not depend on $t$, we have
  \[ (1 + C t^{-5/3})^t = e^{C t^{-2/3}}[1 + o(1)]\]
which tends to $1$ as $t \to \infty$.  Since $\frac{d^4}{3}$ does not depend on $t$, it follows that $\left[1 - \frac{1}{3}(d \delta)^4 \right]^t \to 1$ and $\left[1 + \frac{1}{3}(d \delta)^4 \right]^t \to 1$ as $t \to \I$.  Since $t \delta^2 = t^{1/6}$ and $D_1, D_2, d$ do not depend on $t$, it follows that $[ 1 - e^{-\frac{1}{2} t(D_1 \delta)^2}]^{(d-1)/2} \to 1$ and $[1 - e^{-t (2 D_2 \delta)^2}]^{(d-1)/2} \to 1$ as $t \to \I$.  Finally, for $C$ that does not depend on $t$ we have
  \[ ( 1 + C t^{-5/2} )^{t} = e^{C t^{-3/2}}[1 + o(1)]\]
which tends to $1$ as $t \to \I$; since $\frac{d^6}{4}$ does not depend on $t$, it follows that $\left[ 1 + \frac{1}{4} (d \delta)^6 \right]^{t/2} \to 1$ as $ t \to \I$.  

Putting the above pieces together, we have now shown that as $t \to \I$, $L(v, k, t, t^{-5/12}) \to 1$ and $U(v, k, t, t^{-5/12}) \to 1$.  Hence, \eqref{return_prob_upper} and \eqref{return_prob_lower} imply that if $t$ is such that $t \frac{k}{v} \in \Z$ and $t \frac{k(k-1)}{v(v-1)} \in \Z$, 
\begin{align*}
  & \liminf_{ t \to \I} \frac{\PB_{v,k}^{(t)}(\vec 0, \vec 0)}{\left[ \frac{(k-1)^{v-1}}{\sqrt{(2 \pi t)^{d-1} \det(M)}} \right]} \geq \lim_{t \to \I} \left[ L(v, k, t, t^{-\frac{5}{12}}) - \frac{e^{-\nk^{-1} \frac{11}{768} t^{1/6}}}{\left[ \frac{(k-1)^{v-1}}{\sqrt{(2 \pi t)^{d-1} \det(M)}} \right]} \right] = 1
\end{align*}
and
\begin{align*}
  & \limsup_{ t \to \I} \frac{\PB_{v,k}^{(t)}(\vec 0, \vec 0)}{\left[ \frac{(k-1)^{v-1}}{\sqrt{(2 \pi t)^{d-1} \det(M)}} \right]} \leq \lim_{t \to \I} \left[ U(v, k, t, t^{-\frac{5}{12}}) + \frac{e^{-\nk^{-1} \frac{11}{768} t^{1/6}}}{\left[ \frac{(k-1)^{v-1}}{\sqrt{(2 \pi t)^{d-1} \det(M)}} \right]} \right] = 1.
\end{align*}
Combining these inequalities with \eqref{prob_number_relation} and the calculation of $\det(M)$ in Lemma \ref{determinant_calc} completes the proof.
\end{proof}

\section{Conclusion}


The basic strategy of this work is adopted in principle from \cite{wdl_levin}, where these analogous tasks were completed for partial Hadamard matrices instead of BIBD incidence matrices.  However, the structural differences between the two combinatorial design types necessitated two significant adaptations. First, the maximal set of the Hadamard walk characteristic function is rather different from the one given for the BIBD walk characteristic function in \eqref{lambda_structure}.  In particular, the maximal set for the partial Hadamard walk characteristic function was a zero-dimensional subset of $\R^d$, whereas the maximal set for the BIBD walk characteristic function was a one-dimensional subset of $\R^{d}$.  This corresponds to the fundamental difference that the partial Hadamard walk was supported on a $d$-dimensional sublattice of $\R^{d}$, whereas the BIBD walk is actually supported on an $\left(d -1 \right)$-dimensional sublattice of $\R^{d}$.

The second key difference between the partial Hadamard walk and the BIBD walk rested in a computation of a second moment.  Specifically, finding the return probabilities of each walk required computation of the quantity \mbox{$\E[ ( \vec \mu \cdot Y_1 )^2 ]$}, where $\vec \mu \in \R^{d}$ and $Y_1$ represented a single step of the respective random walks.  In both cases, it was computed that $\E[(\vec \mu \cdot Y_1)^2] = \vec \mu^T N \vec \mu$ for some $d \times d$ matrix $N$.  In the BIBD walk, $N$ was the combinatorially-defined matrix given in \eqref{matrix_coeff_def}, which required significant further analysis and a lengthy computation of its principal minor.  In the partial Hadamard walk, $N$ was instead the identity matrix $I_d$, which simplified many of the calculations and entirely avoided the need for a discussion such as that in Section \ref{S:matrix_determinant}.

While we believe that counting the incidence matrices of BIBDs is of independent interest, we recall here that the more famous and well-studied problem in combinatorial design theory is that of the number of BIBD isomorphism classes. The relationship between the two problems involves involves the difficult combinatorics of permuting the rows and columns of the matrices. We hope at some point to be able to exploit Theorem \ref{number_BIBD_matrices} to gain some insight about the number of underlying isomorphism classes. It may also be possible to sharpen many of the estimates contained throughout in order to solve some currently unanswered existence questions for particular configurations of $v, k, t$, particularly in the case where $t$ is significantly larger than $v$ and $k$.



\section{Acknowledgements}
The author would like to extend his appreciation to David Levin for advice and guidance throughout completion of this project, as well as to Warwick de Launey, whose notes on this subject were helpful.

\bibliography{thesis_montgomery}

\begin{bibdiv}
\begin{biblist}

\bib{abel}{article}{
      author={Abel, R. J.~R.},
       title={Forty-three balanced incomplete block designs},
        date={1994},
        ISSN={0097-3165},
     journal={J. Combin. Theory Ser. A},
      volume={65},
      number={2},
       pages={252\ndash 267},
  url={http://www.sciencedirect.com/science/article/pii/009731659490023X},
}

\bib{batista}{article}{
      author={{Batista}, Milan},
       title={{A note on a generalization of Sherman-Morrison-Woodbury
  formula}},
        date={2008-07},
     journal={ArXiv e-prints},
      eprint={http://arxiv.org/abs/0807.3860v2},
}

\bib{billingsley}{book}{
      author={Billingsley, Patrick},
       title={Probability and measure},
     edition={Third},
   publisher={John Wiley \& Sons Inc.},
     address={New York},
        date={1995},
        note={A Wiley-Interscience Publication},
}

\bib{cr}{article}{
      author={{Chowla}, S.},
      author={{Ryser}, H.~J.},
       title={Combinatorial problems},
        date={1950},
        ISSN={0008-414X},
     journal={Canad. J. Math.},
      volume={2},
       pages={93\ndash 99},
         url={http://dx.doi.org/10.4153/CJM-1950-009-8},
}

\bib{crc}{book}{
      editor={Colbourn, C.~J.},
      editor={Dinitz, Jeffrey~H.},
       title={Handbook of combinatorial designs},
     edition={Second},
      series={Discrete Mathematics and Its Applications},
   publisher={CRC Press},
     address={Boca Raton, FL},
        date={2006},
        ISBN={9781439832349},
         url={http://books.google.com/books?id=S9FA9rq1BgoC},
}

\bib{wdl_levin}{article}{
      author={{de Launey}, Warwick},
      author={Levin, David~A.},
       title={A {F}ourier-analytic approach to counting partial {H}adamard
  matrices},
        date={2010},
        ISSN={1936-2447},
     journal={Cryptogr. Commun.},
      volume={2},
      number={2},
       pages={307\ndash 334},
}

\bib{dm}{article}{
      author={Denny, Paul~C.},
      author={Mathon, Rudolf},
       title={A census of $t$-$(t+8,t+2,4)$ designs, $2\leq t \leq 4$},
        date={2002},
        ISSN={0378-3758},
     journal={J. Statist. Plann. Inference},
      volume={106},
      number={1–2},
       pages={5\ndash 19},
  url={http://www.sciencedirect.com/science/article/pii/S0378375802001970},
        note={Experimental Design and Related Combinatorics},
}

\bib{dinitz}{book}{
      editor={Dinitz, Jeffrey~H.},
      editor={Stinson, Douglas~R.},
       title={Contemporary design theory},
   publisher={John Wiley and Sons Ltd},
     address={New York},
        date={1992},
}

\bib{harville}{book}{
      author={Harville, David~A.},
       title={Matrix algebra from a statistician's perspective},
   publisher={Springer},
     address={New York},
        date={1997},
        ISBN={0-387-94978-X},
}

\bib{htjl}{article}{
      author={Houghten, S.~K.},
      author={Thiel, L.~H.},
      author={Janssen, J.},
      author={Lam, C. W.~H.},
       title={There is no $(46, 6, 1)$ block design},
        date={2001},
        ISSN={1520-6610},
     journal={J. Combin. Des.},
      volume={9},
      number={1},
       pages={60\ndash 71},
  url={http://dx.doi.org/10.1002/1520-6610(2001)9:1<60::AID-JCD5>3.0.CO;2-W},
}

\bib{hwang}{article}{
      author={Hwang, Suk-Geun},
       title={Cauchy's interlace theorem for eigenvalues of {H}ermitian
  matrices},
        date={2004},
        ISSN={00029890},
     journal={Amer. Math. Monthly},
      volume={111},
      number={2},
       pages={157\ndash 159},
         url={http://www.jstor.org/stable/4145217},
}

\bib{klp}{article}{
      author={{Kuperberg}, G.},
      author={{Lovett}, S.},
      author={{Peled}, R.},
       title={{Probabilistic existence of regular combinatorial structures}},
        date={2013-02},
     journal={ArXiv e-prints},
      eprint={http://arxiv.org/abs/1302.4295},
}

\bib{montgomery_thesis}{thesis}{
      author={Montgomery, Aaron},
       title={Topics in random walks},
        type={Ph.D. Thesis},
        date={2013},
}

\bib{morales}{article}{
      author={Morales, Luis~B.},
       title={Constructing difference families through an optimization
  approach: Six new {BIBD}s},
        date={2000},
        ISSN={1520-6610},
     journal={J. Combin. Des.},
      volume={8},
      number={4},
       pages={261\ndash 273},
  url={http://dx.doi.org/10.1002/1520-6610(2000)8:4<261::AID-JCD4>3.0.CO;2-U},
}

\bib{ostergard}{article}{
      author={\"{O}sterg{\aa}rd, Patric R.~J.},
       title={Enumeration of $2$-$(12,3,2)$ designs},
        date={2000},
        ISSN={10344942},
     journal={Australas. J. Combin.},
      volume={22},
       pages={227\ndash 231},
}

\bib{ok}{article}{
      author={\"{O}sterg{\aa}rd, Patric R.~J.},
      author={Kaski, Petteri},
       title={Enumeration of $2$-$(9, 3, \lambda)$ designs and their
  resolutions},
        date={2002},
        ISSN={0925-1022},
     journal={Des. Codes Cryptogr.},
      volume={27},
      number={1-2},
       pages={131\ndash 137},
         url={http://dx.doi.org/10.1023/A\%3A1016558720904},
}

\bib{rowlinson}{book}{
      author={Rowlinson, Peter},
       title={Surveys in combinatorics, 1995.},
      series={London Mathematical Society Lecture Note Series},
   publisher={Cambridge University Press},
     address={Cambridge, England},
        date={1995},
        ISBN={9780521497978},
  url={http://search.ebscohost.com.bw.opal-libraries.org/login.aspx?direct=true&db=nlebk&AN=552651&site=eds-live},
}

\bib{spitzer}{book}{
      author={Spitzer, Frank},
       title={Principles of random walks},
     edition={Second},
   publisher={Springer-Verlag},
     address={New York},
        date={1976},
        note={Graduate Texts in Mathematics, Vol. 34},
}

\bib{vantrung}{article}{
      author={{van Trung}, Tran},
       title={The existence of symmetric block designs with parameters $(41,
  16, 6)$ and $(66, 26, 10)$},
        date={1982},
        ISSN={0097-3165},
     journal={J. Combin. Theory Ser. A},
      volume={33},
      number={2},
       pages={201\ndash 204},
  url={http://www.sciencedirect.com/science/article/pii/0097316582900085},
}

\end{biblist}
\end{bibdiv}

\end{document}